\documentclass[twoside,british,english]{svjour3}
\usepackage[T1]{fontenc}
\usepackage{units}
\usepackage{url}
\usepackage{amsmath}
\usepackage{amssymb}
\usepackage{graphicx}
\usepackage{setspace}
\usepackage[authoryear]{natbib}

\makeatletter
\numberwithin{equation}{section}
\numberwithin{figure}{section}
  \newenvironment{svmultproof}{\begin{proof}}{\qed\end{proof}}
\newenvironment{lyxlist}[1]
{\begin{list}{}
{\settowidth{\labelwidth}{#1}
 \setlength{\leftmargin}{\labelwidth}
 \addtolength{\leftmargin}{\labelsep}
 }}
{\end{list}}

\makeatother

\usepackage{babel}
  \addto\captionsbritish{}
  \addto\captionsbritish{}
  \addto\captionsbritish{}
  \addto\captionsbritish{}
  \addto\captionsbritish{}
  \addto\captionsbritish{}
  \addto\captionsenglish{}
  \addto\captionsenglish{}
  \addto\captionsenglish{}
  \addto\captionsenglish{}
  \addto\captionsenglish{}
  \addto\captionsenglish{}

\begin{document}

\title{Approximation of high quantiles from intermediate quantiles}

\author{Cees de Valk}

\institute{CentER, Tilburg University, P.O. Box 90153, 5000 LE Tilburg, The
Netherlands. Email: ceesfdevalk@gmail.com}

\date{March 4, 2016}
\maketitle
\begin{abstract}
Motivated by applications requiring quantile estimates for very small
probabilities of exceedance $p_{n}\ll1/n$, this article addresses
estimation of high quantiles for $p_{n}$ satisfying $p_{n}\in[n^{-\tau_{2}},n^{-\tau_{1}}]$
for some $\tau_{1}>1$ and $\tau_{2}>\tau_{1}$. For this purpose,
the tail regularity assumption $\log U\circ\exp\in ERV$ (with $U$
the left-continuous inverse of $1/(1-F)$, and $ERV$ the extended
regularly varying functions) is explored as an alternative to the
classical regularity assumption $U\in ERV$ (corresponding to the
Generalised Pareto tail limit). Motivation for the alternative regularity
assumption is provided, and it is shown to be equivalent to a limit
relation for the logarithm of the survival function, the log-GW tail
limit, which \foreignlanguage{british}{generalises} the GW (Generalised
Weibull) tail limit, a generalisation of the Weibull tail limit. The
domain of attraction is described, and convergence results are presented
for quantile approximation and for a simple quantile estimator based
on the log-GW tail. Simulations are presented, and advantages and
limitations of log-GW-based estimation of high quantiles are indicated.\end{abstract}

\begin{quote}
\textbf{Mathematics Subject Classification (20}1\textbf{0)}: 60G70,
62G32, 26A12, 26A48
\end{quote}

\section{Introduction\label{sec:Introduction}}

An important application of extreme value theory is the estimation
of tail quantiles. Theoretical analysis usually addresses tail quantile
estimation from $n$ independent random variables $\{X_{1},...,,X_{n}\}$
with common distribution function $F$, and considers the asymptotic
properties of estimators as $n\rightarrow\infty$. Of particular interest
are \emph{high quantiles}, exceeded with probabilities $p_{n}=O(1/n)$;
see \emph{e.g. }\citet{Weissman}, \citet{Dekkers=000026E=000026dH},
\citet{deH=000026R} and for dependent random variables, \citet{Drees}. 

Let $X_{1,n}\leq X_{2,n}\leq...\leq X_{n,n}$ be the order statistics
derived from $\{X_{1},...,,X_{n}\}$, and let $U$ denote the left-continuous
inverse of $1/(1-F)$ on $(1,\infty)$. The\emph{ }\textit{intermediate
quantile} $U(n/k_{n})$, with the sequence $(k_{n})$ satisfying
\begin{equation}
k_{n}\in\{1,..,n\}\;\forall n\in\mathbb{N},\quad k_{n}/n\rightarrow0\quad\textrm{and}\quad k_{n}\rightarrow\infty,\label{eq:k_n}
\end{equation}
is under certain additional conditions estimated consistently by \textit{\emph{the
intermediate order statistic}} $X_{n-k_{n}+1,n}$ (\emph{e.g.} \citet{Laurens  boek},
Theorem 2.4.1). In contrast, the expected number of data points exceeding
a high quantile is eventually bounded. A high quantile estimator can
therefore not be expected to converge without some form of regularity
of the tail, allowing it to be derived from intermediate order statistics.

The classical regularity assumption on the upper tail of the distribution
function $F$ is often expressed as a condition on $U$; it requires
that a positive function $w$ and a non-constant function $\varphi$
exist such that
\begin{equation}
\lim_{t\rightarrow\infty}\frac{U(t\lambda)-U(t)}{w(t)}=\varphi(\lambda)\quad\forall\lambda\in C_{\varphi},\label{eq:U exc}
\end{equation}
with $C_{\varphi}$ the continuity points of $\varphi$ in $(0,\infty)$.
As the limiting function $\varphi$ is continuous (\emph{e.g.} \citet{Laurens  boek},
Theorem 1.1.3), $U$ satisfying (\ref{eq:U exc}) is extended regularly
varying (see \emph{e.g.} Appendix B2 of \citet{Laurens  boek}, or
Chapter 3 of \citet{Bingham}). Therefore, $w$ can be chosen to be
regularly varying and (since $U$ is nondecreasing) such that 
\begin{equation}
\varphi=h_{\gamma}\label{eq:phi}
\end{equation}
for some real $\gamma$ with for all positive $\lambda$,
\begin{equation}
h_{\gamma}(\lambda):={\scriptstyle \int_{1}^{\lambda}}t^{\gamma-1}dt,\label{eq:B}
\end{equation}
which is $\gamma^{-1}(\lambda^{\gamma}-1)$ if $\gamma\neq0$ and
$\log\lambda$ if $\gamma=0$; (\ref{eq:U exc}) with (\ref{eq:phi})
is equivalent to a Generalised Pareto (GP) tail limit for the survival
function (\emph{e.g.} \citet{Laurens  boek}, Theorem 1.1.2). In (\ref{eq:U exc}),
the limit on the right-hand side was left unspecified in order to
stress the nonparametric nature of the classical regularity assumption,
which makes it particularly attractive from the point of view of applications.

When referring to (\ref{eq:U exc}), we will write $U\in ERV$, with
$ERV$ the extended regularly varying functions%
\footnote{Ignoring that as an assumption, (\ref{eq:U exc}) is formally weaker
than $U\in ERV$; but since $U$ is nondecreasing, we know that $C_{\varphi}=\mathbb{R}^{+}$,
so the difference is immaterial. %
}. We will write $U\in ERV_{S}$ to specify that in addition, (\ref{eq:phi})
holds with $\gamma\in S\subset\mathbb{R}$, and $U\in ERV_{\{\gamma\}}(w)$
for (\ref{eq:U exc}) and (\ref{eq:phi}) with a particular $\gamma$
and positive $w$. We will apply the same notational conventions when
a limit relation of the form (\ref{eq:U exc}) applies to a nondecreasing
function other than $U$. For a regularly varying function $g$ (\emph{e.g.}
\citet{Bingham}), we will write $g\in RV$, or $g\in RV_{S}$ to
specify that $\lim_{t\rightarrow\infty}g(t\lambda)/g(t)=\lambda^{\alpha}$
for all $\lambda>0$ for some $\alpha\in S\subset\mathbb{R}$.

It has been known for long that existence of the limit (\ref{eq:U exc})
alone is of limited value for approximation of a high quantile $U(1/p_{n})$
with $p_{n}=O(1/n)$ from an intermediate quantile $U(n/k_{n})$ with
$(k_{n})$ as in (\ref{eq:k_n}), since $\lambda_{n}:=k_{n}/(np_{n})\rightarrow\infty$
as $n\rightarrow\infty$. Usually, additional assumptions on the rate
of convergence in (\ref{eq:U exc}) are introduced for this purpose,
such as (strong) second-order extended regular variation (\citet{deHaan=000026Stadt},
\citet{Laurens  boek}), the Hall class (\citet{Hall}), or conditions
(1.5) and (1.6) of \citet{deH=000026R}.

In this article, a different approach is explored: instead of strengthening
(\ref{eq:U exc}), we will look for an alternative regularity assumption
specifically to approximate certain high quantiles from intermediate
quantiles, and by extension, to estimate such high quantiles. The
quantiles we will focus on are very high quantiles corresponding to
probabilities of exceedance $(p_{n})$ satisfying
\begin{equation}
p_{n}\in[n^{-\tau_{2}},n^{-\tau_{1}}]\quad\textrm{for some}\quad\tau_{1}>1,\;\tau_{2}>\tau_{1},\label{eq:p_n}
\end{equation}
without excluding that the approximation may also be suitable for
less rapidly vanishing $(p_{n})$. This choice is motivated by applications
requiring quantile estimates for probabilities of exceedance $p_{n}$
satisfying $p_{n}n\ll1$, such as flood hazard assessment (\citet{arch-en}),
design criteria on wind, waves and currents for offshore structures
(\citet{ISO}, paragraph A.5.7), seismic hazard assessment (\citet{Adams})
and analysis of bank operational risk (\citet{Cope}). In such applications,
one would want an estimator for $U(1/p_{n})$ which converges in a
meaningful sense even when $p_{n}n\rightarrow0$ as $n\rightarrow\infty$.
However, the latter condition is difficult to handle in its full generality.
Therefore, we will narrow the focus to $(p_{n})$ satisfying (\ref{eq:p_n}).
Moreover, we will try to find an estimator which for $p_{n}=n^{-\tau}$
converges (in some yet-to-be-defined sense) uniformly in $\tau\in[1,T]$
for every $T>1$. In practical terms, this means that if the assumptions
for convergence are satisfied, then an estimate of a quantile exceeded
with a probability of, say, 0.01 can be extended to an estimate of
the quantile exceeded with a probability of 0.0001 without seriously
stretching the assumptions%
\footnote{Of course, this does not obviate the need to investigate whether these
assumptions apply.%
}, as these probabilities differ only a factor of two in terms of $\tau$.
Such flexibility is important in applications, because $p_{n}$ is
generally based on social and economic considerations, without regard
for the feasibility of estimating $U(1/p_{n})$. 

For convenience, we will assume throughout that $U(\infty):=\lim_{t\rightarrow\infty}U(t)>1$.

\section{An alternative regularity condition\label{sec:Ambition-and}}

The alternative regularity assumption on the upper tail of $F$ proposed
for estimation of a very high quantile $U(1/p_{n})$ with $(p_{n})$
satisfying (\ref{eq:p_n}) is
\begin{equation}
\log q\in ERV\label{eq:logq_in_PI}
\end{equation}
with 

\begin{equation}
q:=U\circ\exp.\label{eq:def_q}
\end{equation}

(\ref{eq:logq_in_PI}) is of the same nonparametric form as the classical
regularity assumption (\ref{eq:U exc}), but with $U$ replaced by
$\log U\circ\exp$. Therefore, it implies that for some real $\theta$
and positive function $g$,
\begin{equation}
\lim_{y\rightarrow\infty}\frac{\log q(y\lambda)-\log q(y)}{g(y)}=h_{\theta}(\lambda)\quad\forall\lambda>0.\label{eq:logq_in_PI_explic}
\end{equation}

To see the relevance of (\ref{eq:logq_in_PI}) for approximation of
a very high quantile $q(-\log p_{n})=U(1/p_{n})$ with $(p_{n})$
satisfying (\ref{eq:p_n}) from an intermediate quantile $q(\log(n/k_{n}))=U(n/k_{n})$,
assume that in addition to (\ref{eq:k_n}), $\limsup_{n\rightarrow\infty}\log k_{n}/\log n<1.$
This ensures that $-\log p_{n}=O(\log(n/k_{n}))$ as $n\rightarrow\infty$,
and since convergence in (\ref{eq:logq_in_PI_explic}) is locally
uniform in $\lambda>0$ (\emph{e.g.} \citet{Bingham}, Theorem 3.1.16),
it implies
\[
\lim_{n\rightarrow\infty}\left|\frac{\log U(1/p_{n})-\log U(n/k_{n})}{g(\log(n/k_{n}))}-h_{\theta}\left(\frac{\log(1/p_{n})}{\log(n/k_{n})}\right)\right|=0.
\]

The limit relation (\ref{eq:logq_in_PI_explic}) can be reformulated
in terms of the survival function:
\begin{theorem}
\label{thm:BdeHP} The limit relation (\ref{eq:logq_in_PI_explic})
for some positive function $g$ and real $\theta$ is equivalent to\textup{
\begin{equation}
\lim_{y\rightarrow\infty}\frac{\log(1-F(q(y)\textrm{e}^{xg(y)}))}{y}=-h_{\theta}^{-1}(x)\quad\forall x\in h_{\theta}(\mathbb{R}^{+}),\label{eq:log-GW_surv_logform}
\end{equation}
}\end{theorem}
\begin{svmultproof}
Equivalence of (\ref{eq:logq_in_PI_explic}) and (\ref{eq:log-GW_surv_logform})
is implied by Lemma 1.1.1 in \citet{Laurens  boek}. 
\end{svmultproof}

The pair (\ref{eq:logq_in_PI_explic}) and (\ref{eq:log-GW_surv_logform})
of equivalent limits can be seen as the analogue for $\log q$ of
(\ref{eq:U exc}) with $\varphi=h_{\gamma}$ and the equivalent GP
limit for the survival function
\begin{equation}
\lim_{t\rightarrow\infty}t(1-F(xw(t)+U(t)))=1/h_{\gamma}^{-1}(x)\quad\forall x\in h_{\gamma}(\mathbb{R}^{+})\label{eq:GP_surv}
\end{equation}
(\emph{e.g.} \citet{Laurens  boek}, Theorem 1.1.2). It is important
to realise that convergence of a log-ratio of probabilities as in
(\ref{eq:log-GW_surv_logform}) is a much weaker notion than convergence
of a ratio of probabilities as in the GP limit (\ref{eq:GP_surv}).
This difference reflects precisely the difference in extrapolation
range between (\ref{eq:logq_in_PI_explic}) and (\ref{eq:U exc}):
when extrapolating over a longer range, larger errors should be expected
in principle, unless additional assumptions apply. 

To illustrate that condition (\ref{eq:logq_in_PI}) is a natural assumption,
Proposition \ref{pro:spec-lim} below shows how it may arise in the
context of a GP tail limit (\ref{eq:U exc}) with (\ref{eq:phi})
and the GP quantile approximation
\begin{equation}
\tilde{U}_{t}(z):=U(t)+h_{\gamma}(z/t)w(t)\label{eq:Utilde_def}
\end{equation}
with $w$ and $\gamma$ as in (\ref{eq:U exc}) and (\ref{eq:phi}).
\begin{proposition}
\label{pro:spec-lim}Let $U\in ERV_{\{\gamma\}}(w)$.

(a) If $\gamma>0$, then
\begin{equation}
\lim_{t\rightarrow\infty}\frac{\log\tilde{U}_{t}(t^{\lambda})-\log U(t^{\lambda})}{\log U(t^{\lambda})-\log U(t)}=0\quad\forall\lambda>1\label{eq:GP_conv_log}
\end{equation}
and
\[
\log q\in ERV_{\{1\}}.
\]

(b) If $\gamma=0$ and 
\begin{equation}
\lim_{t\rightarrow\infty}\frac{\tilde{U}_{t}(t^{\lambda})-U(t^{\lambda})}{U(t^{\lambda})-U(t)}=0\quad\forall\lambda>1,\label{eq:GP_conv}
\end{equation}
then
\[
q\in ERV_{\{1\}}\quad\textrm{and}\quad\log q\in ERV_{\{0\}}(1).
\]
\end{proposition}
\begin{proof}
The proof is found in Subsection \ref{sub:Proof-of-spec-lim}.
\end{proof}

For distribution functions in the domain of attraction of the GP tail
limit, Proposition \ref{pro:spec-lim} shows that the condition (\ref{eq:logq_in_PI})
must hold if $\gamma>0$; if $\gamma=0$, it is a necessary condition
for convergence of the relative error in the GP approximation in the
sense of (\ref{eq:GP_conv})%
\footnote{Irrespective of which additional assumptions are invoked in order
to guarantee (\ref{eq:GP_conv}).%
}. 

Proposition \ref{pro:spec-lim} also provides some basic insight into
the strengths and limitations of the GP quantile approximation (\ref{eq:Utilde_def})
for very high quantiles. If $\gamma>0$, there is no problem; the
notion of convergence in (\ref{eq:GP_conv_log}) may be weak, but
can be considered appropriate for these heavy-tailed distribution
functions. However, if $\gamma=0$ and (\ref{eq:GP_conv}) holds,
then necessarily, $q\in ERV_{\{1\}}$, which is a restrictive condition.
For example, for the normal distribution, $q\in ERV_{\{1/2\}}$, so
(\ref{eq:GP_conv}) cannot hold. 

In analogy to (\ref{eq:U exc}), a natural generalisation of $q\in ERV_{\{1\}}$
would be $q\in ERV$, so for some real $\theta$ and some positive
function $g$,
\begin{equation}
\lim_{y\rightarrow\infty}\frac{q(y\lambda)-q(y)}{g(y)}=h_{\theta}(\lambda)\quad\forall\lambda>0.\label{eq:GW_explic}
\end{equation}

By a slight modification of Theorem \ref{thm:BdeHP}, (\ref{eq:GW_explic})
is equivalent to
\begin{equation}
\lim_{y\rightarrow\infty}\frac{\log(1-F(xg(y)+q(y)))}{y}=-h_{\theta}^{-1}(x)\quad\forall x\in h_{\theta}(\mathbb{R}^{+}).\label{eq:F_exp_logform}
\end{equation}

Furthermore, if $\theta>0$, then $q\in RV_{\{\theta\}}$ (\citet{Laurens  boek},
Theorem B.2.2(1)) and we may take $\theta q$ for $g$ in (\ref{eq:F_exp_logform}),
resulting in
\begin{equation}
\lim_{y\rightarrow\infty}\frac{\log(1-F(xq(y)))}{y}=-x^{1/\theta}\quad\forall x>0.\label{eq:wbl limit}
\end{equation}

The equivalent limit relations $q\in RV_{\{\theta\}}$ and (\ref{eq:wbl limit})
with $\theta>0$ are known as the Weibull tail limit; see \emph{e.g.}
\citet{Broniatowski}, \citet{Kluppelberg}, \citet{Gardes2011} and
references in the latter. Therefore, we will refer to both (\ref{eq:GW_explic})
and (\ref{eq:F_exp_logform}) as the \emph{Generalised Weibull} (GW)
tail limit. Among the distribution functions with a GW tail limit
are the Weibull, gamma, and normal distributions, but also lighter-tailed
distribution functions satisfying $q\in ERV_{(-\infty,0]}$. The latter
satisfy $\lim_{y\rightarrow\infty}q(y\xi)/q(y)=1$ for all $\xi>1$;
if $q\in ERV_{(-\infty,0)}$, then $q(\infty)$ is finite.

In view of the above, we will refer to (\ref{eq:logq_in_PI_explic})
and (\ref{eq:log-GW_surv_logform}) as the \emph{log-GW} tail limit.
Just as $q\in ERV$ generalises the condition $q\in ERV_{\{1\}}$
arising in the context of a GP limit and GP quantile approximation
in Proposition \ref{pro:spec-lim}(b), we can see $\log q\in ERV$
as a natural generalisation of the restrictive conditions $\log q\in ERV_{\{1\}}$
and $\log q\in ERV_{\{0\}}(1)$ in Proposition \ref{pro:spec-lim}(a)
and (b), respectively. Furthermore, the log-GW tail limit generalises
the GW tail limit: \textit{\emph{if $F$ satisfies }}\textit{$q\in ERV_{\{\theta\}}$}\textit{\emph{,
then it must also satisfy}}%
\footnote{As a reminder, we are always assuming that\textit{ $U(\infty)>1$.}%
}\textit{\emph{ $\log q\in ERV_{\{\min(\theta,0)\}}$; see }}\textit{e.g.}\textit{\emph{
\citet{Dekkers=000026E=000026dH} (Lemma 2.5) and Lemma \ref{lem:f-logf}(a)
in Subsection \ref{sub:Lemma}, included for convenience. Therefore,
the log-GW tail limit is the more important limit relation to consider
as regularity assumption. Nevertheless, the GW limit may be useful
}}in certain applications involving distribution functions with moderate
or light tails\textit{\emph{. In particular, if $\theta<0$, then
$\log q\in ERV_{\{\theta\}}$ if and only if $q\in ERV_{\{\theta\}}$;
see Lemma \ref{lem:f-logf}(c) in Subsection \ref{sub:Lemma}.}}

The following result supplements Proposition \ref{pro:spec-lim} by
describing the possible overlap of the domain of attraction of the
GP limit with the domains of attraction of the GW and log-GW limits.
It just states the plain results; an interpretation follows. 
\begin{theorem}
\label{thm:GW_in_gamma=00003D0} For $q:=U\circ\exp$,\end{theorem}
\begin{lyxlist}{00.00.0000}
\item [{\textit{(a)}}] \textit{If $U\in ERV$ and $q\in ERV$, then $U\in ERV_{\{0\}}$.}
\item [{\textit{(b)}}] \textit{If $U\in ERV$ and $\log q\in ERV$, then}

\begin{lyxlist}{00.00.0000}
\item [{\textit{either}}] \textit{(i) $U\in ERV_{\{0\}}\; and\;\log q\in ERV_{(-\infty,1]},$ }
\item [{\textit{or}}] \textit{(ii) $U\in ERV_{(0,\infty)}\; and\;\log q\in ERV{}_{\{1\}}.$ }
\end{lyxlist}
\end{lyxlist}
\begin{proof}
See Subsection \ref{sub:proof-GW_in_gamma=00003D0}. 
\end{proof}

Theorem \ref{thm:GW_in_gamma=00003D0}(a) supplements Proposition
\ref{pro:spec-lim}(b) for the $\gamma=0$ case: the existence of
a GW limit excludes distribution functions with heavy and light GP
tail limits. Theorem \ref{thm:GW_in_gamma=00003D0}(b) identifies
which specific log-GW limits may coexist with a GP limit. Case (ii)
is the classical Pareto limit encountered in Proposition \ref{pro:spec-lim}(a).
Case (i) concerns lighter tails; note that it is possible that $U\in ERV_{\{0\}}$
and $\log q\in ERV_{\{1\}}$, an example being $q(y)=\exp(y/\log(y+1)-1)$.
By assertion (b), a GP limit with $\gamma<0$ excludes a log-GW limit. 

The domain of attraction of the log-GW limit covers a wide range of
tail behaviour. It includes the domain of attraction of the GW limit
described earlier, and the domain of attraction of the Pareto limit
with $\gamma>0$, but also the distribution functions satisfying $\log q\in ERV_{(0,1)}$,
with tails heavier than a Weibull tail but lighter than a Pareto tail.
As such, it achieves a ``unification'' of the Pareto and Weibull
tail limits sought in \citet{Gardes2011}. An example is the lognormal
distribution, which satisfies $\log q\in ERV_{\{\nicefrac{1}{2}\}}$;
neither (\ref{eq:GP_conv}), nor (\ref{eq:GP_conv_log}) holds for
this distribution function. Finally, the domain of attraction of the
log-GW limit also includes the very heavy-tailed distribution functions
satisfying $\log q\in ERV_{(1,\infty)}$, which do not have classical
limits. For these, the mean of the excess $(X-\alpha)\vee0$ over
any finite threshold $\alpha$ is infinite. 

Having now established the log-GW limit as a widely applicable regularity
assumption for approximation of high quantiles with probabilities
$(p_{n})$ satisfying (\ref{eq:p_n}), the following sections will
address the use of a log-GW tail as model for quantile approximation
and estimation.

\section{\label{sec:approx and error}Approximation and convergence}

The log-GW limit suggests to approximate a quantile $q(z)$ for $z>0$
by $q(y)\textrm{e}^{g(y)h_{\theta}(z/y)}$ for $y\in q^{-1}((0,\infty))$
and with $g$ and $\theta$ as in (\ref{eq:logq_in_PI_explic}). As
an introduction to the quantile estimator presented in the next section,
we will consider the following somewhat more general log-GW quantile
approximation:
\begin{equation}
\tilde{q}_{y}(z):=q(y)\textrm{e}^{\tilde{g}(y)h_{\tilde{\theta}(y)}(z/y)},\label{eq:logGW_tailmodel}
\end{equation}
with $\tilde{\theta}$ a real function and $\tilde{g}$ a positive
function, related to $q$ as follows: for some $\xi>1$,
\begin{equation}
\tilde{\theta}(y)-a_{\xi}(y)\rightarrow0\quad\textrm{and}\quad\tilde{g}(y)\sim(\log q(y\xi)-\log q(y))/h_{\tilde{\theta}(y)}(\xi)\quad\textrm{as}\; y\rightarrow\infty\label{eq:logGW_tailmodelassump}
\end{equation}
with for every $\iota\in(0,1)\cup(1,\infty)$,
\begin{equation}
a_{\iota}(y):=\frac{\log\left|\log q(y\iota^{2})-\log q(y\iota)\right|-\log\left|\log q(y\iota)-\log q(y)\right|}{\log\iota}.\label{eq:al_def}
\end{equation}

If $q$ has a second derivative $q''$, then $a_{\iota}(y)$ may be
regarded as a finite-difference approximation of $y(\log(y(\log q(y))'))'=1+yq''(y)/q'(y)-yq'(y)/q(y)$,
a scale-invariant measure of curvature. 

If $\log q\in ERV_{\{\theta\}}(g)$, then $\log q(\textrm{Id}\cdot\xi)-\log q\in RV_{\{\theta\}}$
for every $\xi>1$ and (\ref{eq:logGW_tailmodelassump}) is equivalent
to $\tilde{\theta}(y)\rightarrow\theta$ and $\tilde{g}(y)\sim g(y)$
as $y\rightarrow\infty$. The following is a straightforward consequence: 
\begin{proposition}
\label{pro:central}If $\log q\in ERV_{\{\theta\}}(g)$ and the real
function $\tilde{\theta}$ and positive function $\tilde{g}$ satisfy
(\ref{eq:logGW_tailmodelassump}), then $\tilde{q}_{y}$ defined by
(\ref{eq:logGW_tailmodel}) satisfies 
\begin{equation}
\lim_{y\rightarrow\infty}\sup_{\lambda\in[\Lambda^{-1},\Lambda]}\left|\frac{\log\tilde{q}_{y}(y\lambda)-\log q(y\lambda)}{g(y)}\right|=0\quad\forall\Lambda>1,\label{eq:conv_ql_rr_log}
\end{equation}
and if
\begin{equation}
\limsup_{y\rightarrow\infty}g(y)<\infty\label{eq:g_evt_bnd}
\end{equation}
(for example, if $q\in ERV$), then in addition, 
\begin{equation}
\lim_{y\rightarrow\infty}\sup_{\lambda\in[\Lambda^{-1},\Lambda]}\left|\frac{\tilde{q}_{y}(y\lambda)-q(y\lambda)}{q(y)g(y)}\right|=0\quad\forall\Lambda>1.\label{eq:conv_ql_rr}
\end{equation}
\end{proposition}
\begin{proof}
A proof of this standard result can be found in Subsection \ref{sub:proof-procentral}.\end{proof}

\begin{remark}
\label{rem:R1}Eq. (\ref{eq:conv_ql_rr_log}) remains valid when $g(y)$
in the denominator is replaced by $\log q(y)$ or by $\log q(y\xi)-\log q(y)$
for any $\xi\in(0,\infty)\setminus\{1\}$, because by (\ref{eq:logq_in_PI_explic}),
\begin{equation}
g(y)/\bigl|\log q(y\xi)-\log q(y)\bigr|=O(1)\label{eq:g_bound}
\end{equation}
as $y\rightarrow\infty$, and therefore also $g(y)/\log q(y)=O(1)$.

Condition (\ref{eq:g_evt_bnd}) implies that $\theta\leq0$ in (\ref{eq:logq_in_PI_explic})
and therefore, that $q$ is of bounded increase (see \citet{Bingham},
Section 2.1); \emph{vice versa}, bounded increase of $q$ implies
(\ref{eq:g_evt_bnd}) by (\ref{eq:logq_in_PI_explic}). If (\ref{eq:g_evt_bnd})
holds, then (\ref{eq:conv_ql_rr_log}) and (\ref{eq:conv_ql_rr})
remain valid when $g(y)$ is replaced by $1$. Furthermore, $q(y)g(y)$
in (\ref{eq:conv_ql_rr}) can be replaced by $q(y\xi)-q(y)$ for any
$\xi\in(0,\infty)\setminus\{1\}$; see Subsection \ref{sub:notes_Remark}.
Furthermore, if $q(\infty)<\infty$, then we may also replace $q(y)g(y)$
in (\ref{eq:conv_ql_rr}) by $q(\infty)-q(y\eta)$ for any $\eta>0$;
see Subsection \ref{sub:notes_Remark}. 
\end{remark}
The normalisation of the quantile approximation error in (\ref{eq:conv_ql_rr_log})
is model-dependent. Whether (\ref{eq:conv_ql_rr}) is applicable,
and which model-independent normalisations may be substituted for
$g$ in (\ref{eq:conv_ql_rr_log}) and (\ref{eq:conv_ql_rr}), depends
on tail weight: \emph{i.e.}, on whether $q$ is of bounded increase;
see Remark \ref{rem:R1}. As an alternative, the error in a quantile
approximation may be expressed in terms of a mismatch between the
probabilities of exceedance of the quantile and of its approximation.
As we will see shortly, this can be done in such a way that a single
model-independent notion of convergence holds if $\log q\in ERV$. 

There may also be other reasons for considering probability-based
quantile approximation and estimation errors. For example, in the
context of structural reliability analysis and safety engineering
(\textit{e.g.} flood protection, tall buildings, bridges, offshore
structures, etc.), the required overall safety level constrains a
design; usually, it takes the form of a maximum tolerated failure
rate, fixed in legislation or in rules issued by regulators or classification
societies. Within this context, errors in estimates of load quantiles
are often viewed in terms of equivalent errors in frequency of exceedance. 

\begin{singlespace}
In the present context, a natural expression of the mismatch between
$1-F(\tilde{q}_{y}(z))$ and $1-F(q(z))$ is
\begin{equation}
\tilde{\nu}_{y}(z):=\frac{q^{-1}(\tilde{q}_{y}(z))}{q^{-1}(q(z))}-1=\frac{\log(1-F(\tilde{q}_{y}(z)))}{\log(1-F(q(z)))}-1.\label{eq:nu_def}
\end{equation}

\end{singlespace}

Because $F$ may be constant over some interval, it is possible that
$\tilde{\nu}_{y}(z)=0$ while $\tilde{q}_{y}(z)>q(z)$. If $q(\infty)<\infty$
and $\tilde{q}_{y}(z)>q(\infty)$, then $\tilde{\nu}_{y}(z)=\infty$.
If $F$ is continuous, then $-\log(1-F(q(z)))=q^{-1}(q(z))=z$ in
(\ref{eq:nu_def}). 

For the log-GW approximation (\ref{eq:logGW_tailmodel}), convergence
of $\tilde{\nu}_{y}(y\lambda)$ to zero as $y\rightarrow\infty$ for
$\lambda>0$ is a similarly weak notion of convergence as convergence
to the log-GW limit in (\ref{eq:log-GW_surv_logform}). In fact, if
$F$ is continuous, then with $\tilde{\theta}(y)=\theta$ and $\tilde{g}(y)=g(y)$
in (\ref{eq:logGW_tailmodel}), the log-GW limit can be written alternatively
as $\lim_{y\rightarrow\infty}\lambda\tilde{\nu}_{y}(y\lambda)=0$
for all $\lambda>0$. A somewhat more general result is the following.
\begin{theorem}
\label{thm: PHI-omega-nu}If $\log q\in ERV$ and real functions $\tilde{\theta}$
and $\tilde{g}$, $\tilde{g}$ positive, satisfy (\ref{eq:logGW_tailmodelassump}),
then $\tilde{q}_{y}$ defined by (\ref{eq:logGW_tailmodel}) satisfies
\begin{equation}
\lim_{y\rightarrow\infty}\sup_{\lambda\in[\Lambda^{-1},\Lambda]}\left|\tilde{\nu}_{y}(y\lambda)\right|=0\quad\forall\Lambda>1.\label{eq:conv_nul}
\end{equation}
\end{theorem}
\begin{proof}
See Subsection \ref{sub:ProofProPHI}. 
\end{proof}

Alternatively, one may want to consider a stronger notion of convergence
such as
\begin{equation}
\lim_{y\rightarrow\infty}\frac{1-F(q(y\lambda))}{1-F(\tilde{q}_{y}(y\lambda))}=1\qquad\forall\lambda\geq1.\label{eq:strong-pconv}
\end{equation}

If $\lim_{y\rightarrow\infty}y^{-1}\log(1-F(q(y)))=-1$%
\footnote{This very weak condition is ensured by, for example, (\ref{eq:log-GW_surv_logform}),
or (\ref{eq:GP_surv}), or continuity of $F$.%
}, then by taking the logarithm, (\ref{eq:strong-pconv}) can be seen
to be equivalent to $\lim_{y\rightarrow\infty}y\tilde{\nu}_{y}(y\lambda)=0$
for all $\lambda\geq1$.\emph{ }Ensuring this convergence rate condition
requires strengthening of the assumption of a log-GW limit. We will
discuss this further within the context of a specific estimator in
the next section.

\section{\label{sec:estimators}A simple high quantile estimator}

To demonstrate the potential of the alternative regularity condition
for estimation of high quantiles, this section introduces a quantile
estimator closely related to the log-GW approximation (\ref{eq:logGW_tailmodel})
and presents consistency results. 

Consider a sequence of independent random variables $(X_{n})$ with
$X_{i}\sim F$ for all $i\in\mathbb{N}$. Let $X_{k,n}$ denote the
$k$-th lowest order statistic out of $\{X_{1},..,X_{n}\}$. Let $\iota>1$
be fixed, and let%
\footnote{For notational convenience, we write some sequences as functions on
$\mathbb{N}$.%
} $k_{2}:\mathbb{\mathbb{N}\rightarrow N}$ be nondecreasing and such
that $k_{2}(n)\in\{1,...,n-1\}$ for all $n\in\mathbb{N}$. Define
for $j\in\{0,1\}$,
\begin{equation}
k_{j}(n):=\left\lfloor (k_{2}(n)/n)^{\iota^{j-2}}n\right\rfloor .\label{eq:k_j}
\end{equation}

A simple log-GW-based estimator for a quantile $q(z)$ with probability
of exceedance $\textrm{e}^{-z}$ is $\hat{q}_{n}(z)$, defined for
every $z>0$ and $n\in\mathbb{N}$ such that $X_{n-k_{0}(n)+1,n}>0$
by
\begin{equation}
\hat{q}_{n}(z):=X_{n-k_{0}(n)+1,n}\exp\left(\hat{g}_{n}h_{\hat{\theta}_{n}}(z/y_{n})\right)\label{eq:def_qhat}
\end{equation}
with
\begin{equation}
\hat{\theta}_{n}:=\frac{\log\log\frac{X_{n-k_{2}(n)+1,n}}{X_{n-k_{1}(n)+1,n}}-\log\log\frac{X_{n-k_{1}(n)+1,n}}{X_{n-k_{0}(n)+1,n}}}{\log\iota},\label{eq:def_ahat}
\end{equation}
\begin{equation}
\hat{g}_{n}:=\frac{\log\frac{X_{n-k_{1}(n)+1,n}}{X_{n-k_{0}(n)+1,n}}}{h_{\hat{\theta}_{n,\iota}}(\iota)},\label{eq:def_ghat}
\end{equation}
and
\begin{equation}
y_{n}:=\log(n/k_{0}(n)).\label{eq:yn_def}
\end{equation}

This estimator can be regarded as a straightforward application of
the approximation (\ref{eq:logGW_tailmodel}) to the sampling distribution
of $\{X_{1},..,X_{n}\}$ instead of $F$, taking
\begin{equation}
g_{\iota}(y):=(\log q(y\iota)-\log q(y))/h_{a_{\iota}(y)}(\iota)\label{eq:g_i}
\end{equation}
for $\tilde{g}(y)$ and $a_{\iota}(y)$ for $\tilde{\theta}(y)$.
Assume that $k_{2}(n)/n\rightarrow0$ and $k_{2}(n)\rightarrow\infty$
as $n\rightarrow\infty$. Then by (\ref{eq:k_j}), as $\iota>1$,
also $k_{j}(n)/n\rightarrow0$ and $k_{j}(n)\rightarrow\infty$ as
$n\rightarrow\infty$ for $j=1$ and $j=0$. Moreover, if $k_{2}$
is chosen to satisfy
\begin{equation}
\limsup_{n\rightarrow\infty}\frac{\log k_{2}(n)}{\log n}=:c<1,\label{eq:k2_cond}
\end{equation}
then by (\ref{eq:k_j}), $\limsup_{n\rightarrow\infty}(\log k_{0}(n))/\log n=1+\iota^{-2}(c-1)$,
so
\begin{equation}
\liminf_{n\rightarrow\infty}y_{n}/\log n=(1-c)\iota^{-2}.\label{eq:liminf_y}
\end{equation}

Therefore, for every $T\geq1$, eventually
\begin{equation}
[T^{-1}\log n,T\log n]\subset[\lambda^{-1}y_{n},\lambda y_{n}]\quad\forall\lambda>T\iota^{2}/(1-c),\label{eq:yn_over_logn}
\end{equation}
and as a result, $-\log p_{n}$ with $(p_{n})$ as in (\ref{eq:p_n})
is eventually in the interval $[\lambda^{-1}y_{n},\lambda y_{n}]$
for some $\lambda>1$. 

If $\log X$ were replaced by $X$ in (\ref{eq:def_qhat})-(\ref{eq:def_ghat})
and (\ref{eq:k_j}) were modified to $k_{j}(n):={\scriptstyle \left\lfloor k_{2}(n)\iota^{2-j}\right\rfloor }$
and (\ref{eq:yn_def}) to $y_{n}:=n/k_{0}(n)$, then with $\iota=2$,
(\ref{eq:def_ahat}) would become the \citet{PickandsIII} estimator
for the extreme value index $\gamma$, and (\ref{eq:def_qhat}) would
become an estimator for $U(z)$. Pickands' estimator is known to be
inaccurate in comparison to other commonly used estimators; see \textit{e.g.}
\citet{Laurens  boek}. The estimator $\hat{q}_{n}$, also based on
only three order statistics, was chosen as an example here because
of its simplicity.

Analogous to $\tilde{\nu}_{y}(z)$ in (\ref{eq:nu_def}), define the
probability-based quantile estimation error
\begin{equation}
\hat{\nu}_{n}(z):=\frac{q^{-1}(\hat{q}_{n}(z))}{q^{-1}(q(z))}-1=\frac{\log(1-F(\hat{q}_{n}(z)))}{\log(1-F(q(z)))}-1.\label{eq:def_nuhat}
\end{equation}

\begin{theorem}
\label{thm:Pickands-like_random}Let $k_{2}:\mathbb{\mathbb{N}\rightarrow N}$
satisfy (\ref{eq:k2_cond}) and $k_{2}(n)/\log\log n\rightarrow\infty$
as $n\rightarrow\infty$. Consider $\hat{q}_{n}$, $\hat{\theta}_{n}$
and $\hat{g}_{n}$ defined by (\ref{eq:k_j})-(\ref{eq:yn_def}) for
some $\iota>1$. If $\log q\in ERV_{\{\theta\}}(g)$, then\textup{
\begin{equation}
\hat{\theta}_{n}\rightarrow\theta\quad and\quad\hat{g}_{n}/g(y_{n})\rightarrow1\quad a.s.\label{eq:conv_ahat}
\end{equation}
}and for every $T>1$ (see (\ref{eq:nu_def})),
\begin{equation}
\sup_{\tau\in[T^{-1},T]}\left|\hat{\nu}_{n}(\tau\log n)\right|\rightarrow0\quad\textrm{a.s.},\label{eq:conv_nuhat}
\end{equation}
\begin{equation}
\sup_{\tau\in[T^{-1},T]}\left|\frac{\log\hat{q}_{n}(\tau\log n)-\log q(\tau\log n)}{g(y_{n})}\right|\rightarrow0\quad a.s.\label{eq:conv_logqlhat}
\end{equation}
and if (\ref{eq:g_evt_bnd}) holds (for example, if $q\in ERV$),
then in addition,
\begin{equation}
\sup_{\tau\in[T^{-1},T]}\left|\frac{\hat{q}_{n}(\tau\log n)-q(\tau\log n)}{q(y_{n})g(y_{n})}\right|\rightarrow0\quad a.s.\label{eq:conv_qlhat}
\end{equation}
\end{theorem}
\begin{proof}
The proof is found in Subsection \ref{sub:Proof-of-Pickands-like}.
\end{proof}

Theorem \ref{thm:Pickands-like_random} establishes almost sure convergence
of very high quantile estimates for probabilities of exceedance of
$n^{-\tau}$ uniformly for all $\tau$ in an arbitrary compact subset
of $(0,\infty)$ if $\log q\in ERV$. 
\begin{remark}
Remark \ref{rem:R1} about the normalisation in (\ref{eq:conv_ql_rr_log})
and (\ref{eq:conv_ql_rr}) carries over to (\ref{eq:conv_logqlhat})
and (\ref{eq:conv_qlhat}). 
\end{remark}
For the analysis of the asymptotic distributions of errors, the assumption
$\log q\in ERV_{\{\theta\}}$ in Theorem \ref{thm:Pickands-like_random}
will be strengthened somewhat. We assume that the derivative $q'$
of $q$ exists, and 
\begin{equation}
(\log q)'=q'/q\in RV_{\{\theta-1\}},\label{eq:q_diff_cond}
\end{equation}
which implies $\log q\in ERV_{\{\theta\}}(\bar{g})$ with
\begin{equation}
\bar{g}(y)=yq'(y)/q(y).\label{eq:gbar}
\end{equation}

If it is given that $\log q\in ERV_{\{\theta\}}$ and that $q$ is
differentiable, several seemingly weak conditions on $q'$ are known
which ensure (\ref{eq:q_diff_cond}); see \emph{e.g.} \citet{Bingham}
(Theorems 1.7.5 and 3.6.10).

Let $g_{\iota}$ be defined by (\ref{eq:g_i}), $a_{\iota}$ be defined
by (\ref{eq:al_def}), and let
\begin{equation}
\kappa_{\theta}(\lambda,\iota):=\frac{\partial(h_{\theta}(\lambda)/h_{\theta}(\iota))}{\partial\theta}=\begin{cases}
\frac{1}{\iota^{\theta}-1}\left(\lambda^{\theta}\log\lambda-\frac{\lambda^{\theta}-1}{\iota^{\theta}-1}\iota^{\theta}\log\iota\right) & \textrm{if\:}\theta\neq0\\
\frac{1}{2}\log\lambda\left(\frac{\log\lambda}{\log\iota}-1\right) & \textrm{if\:}\theta=0
\end{cases}\label{eq:kappa_def}
\end{equation}
for all real $\theta$, $\iota\in(0,1)\cup(1,\infty)$ and $\lambda>0$;
note that $\kappa_{\theta}(1,\iota)=\kappa_{\theta}(\iota,\iota)=0$.
We will first consider limiting distribution functions of suitably
normalised deviations of the estimates $\hat{\theta}_{n}$, $\hat{q}_{n}$
and $\hat{\nu}_{n}$ from their deterministic analogues.
\begin{theorem}
\label{thm:logGW_normal}If (\ref{eq:q_diff_cond}) holds in addition
to the assumptions for Theorem \ref{thm:Pickands-like_random}, then
\begin{equation}
Z_{n}:=\bigl(\hat{\theta}_{n}-a_{\iota}(y_{n})\bigr)y_{n}\sqrt{k_{2}(n)}h_{\theta}(\iota)\overset{d}{\rightarrow}N(0,(\iota^{\theta-2}/\log\iota)^{2})\label{eq:rho_asnormal}
\end{equation}
and with $\tilde{\theta}=a_{\iota}$ and $\tilde{g}=g_{\iota}$ in
(\ref{eq:logGW_tailmodel}), for all $T>1$,
\begin{equation}
\sup_{z\in[T^{-1}\log n,T\log n]}\left|\bigl(\hat{\nu}_{n}(z)-\tilde{\nu}_{y_{n}}(z)\bigr)y_{n}\sqrt{k_{2}(n)}-\Bigl(\frac{z}{y_{n}}\Bigr)^{-\theta}\kappa_{\theta}\Bigl(\frac{z}{y_{n}},\iota\Bigr)Z_{n}\right|\rightarrow0\quad a.s.\label{eq:nu_normal}
\end{equation}
and 
\begin{equation}
\sup_{z\in[T^{-1}\log n,T\log n]}\left|\frac{\log\hat{q}_{n}(z)-\log\tilde{q}_{y_{n}}(z)}{g(y_{n})}y_{n}\sqrt{k_{2}(n)}-\kappa_{\theta}\Bigl(\frac{z}{y_{n}},\iota\Bigr)Z_{n}\right|\rightarrow0\quad a.s.\label{eq:q_normal}
\end{equation}
for every positive function $g$ satisfying (\ref{eq:logq_in_PI_explic}).\end{theorem}
\begin{proof}
See Subsection \ref{sub:Proof-of-Theorem-GWnormal}.
\end{proof}

Under an additional convergence rate assumption, the previous result
implies asymptotic normality of the estimation errors $\hat{\theta}_{n}-\theta$,
$\hat{\nu}_{n}$ and $\log\hat{q}_{n}-\log q$:
\begin{corollary}
\label{cor:normality}If in addition to the assumptions for Theorem
\ref{thm:logGW_normal}, (\ref{eq:q_diff_cond}) is strengthened to
\begin{equation}
\frac{q'(y\lambda)/q(y\lambda)}{q'(y)/q(y)}=\lambda^{\theta-1}(1+o(1)/\phi(y))\quad as\quad y\rightarrow\infty\quad\forall\lambda>1\label{eq:conv_rate}
\end{equation}
with $\phi$ some positive increasing function satisfying $\lim_{y\rightarrow\infty}\phi(y)/(y\sqrt{\log y})=\infty$,
and if $k_{2}$ satisfies $k_{2}(n)=O(\phi^{2}(y_{n})y_{n}^{-2})$,
then
\begin{equation}
Z_{n}^{0}:=\bigl(\hat{\theta}_{n}-\theta\bigr)y_{n}\sqrt{k_{2}(n)}h_{\theta}(\iota)\overset{d}{\rightarrow}N(0,(\iota^{\theta-2}/\log\iota)^{2}),\label{eq:rho_asnormal-1}
\end{equation}
and for all $T>1$,
\begin{equation}
\sup_{z\in[y_{n},T\log n]}\left|\hat{\nu}_{n}(z)\: y_{n}\sqrt{k_{2}(n)}-\Bigl(\frac{z}{y_{n}}\Bigr)^{-\theta}\kappa_{\theta}\Bigl(\frac{z}{y_{n}},\iota\Bigr)Z_{n}^{0}\right|\rightarrow0\quad a.s.\label{eq:nu_normal-1}
\end{equation}
and
\begin{equation}
\sup_{z\in[y_{n},T\log n]}\left|\frac{\log\hat{q}_{n}(z)-\log q(z)}{g(y_{n})}y_{n}\sqrt{k_{2}(n)}-\kappa_{\theta}\Bigl(\frac{z}{y_{n}},\iota\Bigr)Z_{n}^{0}\right|\rightarrow0\quad a.s.\label{eq:q_normal-1}
\end{equation}
for every positive function $g$ satisfying (\ref{eq:logq_in_PI_explic}).\end{corollary}
\begin{proof}
See Subsection \ref{sub:Proof-of-Theorem-GWnormal}.
\end{proof}

\begin{remark}
Eq. (\ref{eq:nu_normal}) and (\ref{eq:k2_cond}) imply that $y_{n}\sqrt{k_{2}(n)}(\hat{\nu}_{n}(y_{n}\lambda)-\tilde{\nu}_{y_{n}}(y_{n}\lambda))$
is asymptotically normal with zero mean and variance equal to $((\iota^{\theta-2}/\log\iota)\lambda^{-\theta}\kappa_{\theta}\left(\lambda,\iota)\right)^{2}$
for every $\lambda>0$. Similar comments apply to (\ref{eq:q_normal}),
(\ref{eq:nu_normal-1}) and (\ref{eq:q_normal-1}). 
\end{remark}

\begin{remark}
If a function $\phi$ satisfying the conditions of Corollary \ref{cor:normality}
exists, then a $k_{2}$ satisfying $k_{2}(n)=O(\phi^{2}(y_{n})y_{n}^{-2})$,
$k_{2}(n)/\log\log n\rightarrow\infty$ as $n\rightarrow\infty$ and
(\ref{eq:k2_cond}) can always be found; for example, for some $\alpha>0$,
one can take for $k_{2}(n)$ the smallest integer $k$ satisfying
$k\geq\max(1,{\scriptstyle \left\lfloor \min(\textrm{e}^{\alpha y},\phi^{2}(y)y^{-2})\right\rfloor })$
with $y=\iota^{-2}\log(n/k)$. 
\end{remark}

\begin{remark}
\label{rem:strong_conv}Using (\ref{eq:def_nuhat}), it can be seen
that (\ref{eq:nu_normal-1}) implies 
\begin{equation}
\sup_{z\in[y_{n},T\log n]}\biggl|\frac{1-F(q(z))}{1-F(\hat{q}_{n}(z))}-1\biggr|\overset{p}{\rightarrow}0\qquad\forall T>1,\label{eq:strong-conv}
\end{equation}
representing a strong notion of convergence of the probability of
exceedance of the quantile estimate to its target value (this may
be compared to the comment following Theorem \ref{thm: PHI-omega-nu}).
Furthermore, if $g(y)/y$ is eventually bounded as $y\rightarrow\infty$
(so the tail is not heavier than a typical Pareto tail), then (\ref{eq:q_normal-1})
implies that for all $T>1$, $\sup_{z\in[y_{n},T\log n]}\left|\hat{q}_{n}(z)/q(z)-1\right|\overset{p}{\rightarrow}0$.
\end{remark}

Convergence rate assumptions like (\ref{eq:conv_rate}) with $\phi$
some function increasing to infinity are commonly made%
\footnote{Often in the form of a second-order ERV assumption.%
} to derive asymptotic normality of parameter and quantile estimators
under the condition that the rate of increase of $k_{2}$ (or more
in general, the number of upper order statistics controlling the accuracy
of the estimator) is restricted by $\phi$ in some manner. 

For the estimator $\hat{q}_{n}$, the convergence rate assumption
is rather restrictive: corollary \ref{cor:normality} requires that
$\phi(y)/y$ must tend to infinity as $y\rightarrow\infty$. The reason
for this is that each factor $\sqrt{k_{2}(n)}$ in (\ref{eq:rho_asnormal})-(\ref{eq:q_normal})
is preceded by a factor $y_{n}$, which can only increase when reducing
$k_{2}(n)$. While these factors contribute to a low large-sample
variability for this estimator, they make it more difficult or impossible
to ``mask'' bias by reducing $k_{2}(n)$. 

This limitation is due to the particular formulation of this estimator.
Alternative estimators exist which satisfy expressions analogous to
(\ref{eq:rho_asnormal})-(\ref{eq:q_normal}) but without the factors
$y_{n}$, thus weakening the restrictions to be imposed on $\phi$
in (\ref{eq:conv_rate}) for establishing asymptotic normality. For
the special case of a Weibull tail limit, \emph{i.e.}, $\theta=0$
and $g(y)\rightarrow g_{\infty}\in(0,\infty)$ in (\ref{eq:logq_in_PI_explic}),
examples are the estimators for the Weibull tail index $g_{\infty}$
and associated quantile estimators in \citet{G=000026G}. Preliminary
work suggests that this type of estimator may be extended to log-GW
and GW-based quantile estimators under the appropriate tail limits. 

Alternatively, one may try to correct quantile estimates for bias,
which may relax restrictions on $k_{2}$. This would involve extending
the model $\tilde{q}_{y_{n}}$ with $\tilde{\theta}=a_{\iota}$ and
$\tilde{g}=g_{\iota}$ in (\ref{eq:logGW_tailmodel}) and its estimator
$\hat{q}_{n}$ to make $(\log\tilde{q}_{y_{n}}(z)-\log q(y_{n}))/g(y_{n})$
vanish more rapidly with increasing $n$, without substantially slowing
the rate of absolute decrease of $(\log\hat{q}_{n}(z)-\log\tilde{q}_{y_{n}}(z))/g(y_{n})$
in (\ref{eq:q_normal}). Within the context of the GP tail limit and
GP-based high quantile estimation, estimation of a model of second-order
ERV to correct quantile estimates has been developed to an advanced
level; see \emph{e.g.} \citet{Li} and \citet{Cai}. More limited
progress has been made within the context of the Weibull tail limit.
For example, bias correction in \citet{Diebolt_a} can produce asymptotically
normal zero-mean estimation errors with the same variance as obtained
with the asymptotically biased uncorrected estimator%
\footnote{A zero mean value is required for construction of confidence intervals.%
}. These developments suggest that bias correction could be successful
in the context of the log-GW limit and log-GW-based quantile estimation.

\section{\label{sec:simulations}Simulations}

As an illustration, the log-GW-based quantile estimator $\hat{q}_{n}$
defined in (\ref{eq:k_j})-(\ref{eq:yn_def}) was applied with $\iota=2$
to simulated samples of \emph{iid} random variables to estimate very
high quantiles with a probability of exceedance of ${\scriptstyle n^{-2}}$
. For comparison, a GP-based quantile estimator was applied to the
same data. For this purpose, the moment estimator of \citet{Dekkers=000026E=000026dH}
and \citet{deH=000026R} was chosen; see also \citet{Laurens  boek}
(3.5, 4.2 and 4.3.2). With $k:\mathbb{N}\rightarrow\mathbb{N}$ such
that $k(n)\in\{1,..,n\}$ and $X_{n-k(n)+1,n}>0$ for $n$ large enough,
it is given by 
\[
\hat{q}_{n}^{m}(z):=X_{n-k(n)+1,n}+\hat{\sigma}_{n}h_{\hat{\gamma}_{n}^{m}}\left(\textrm{e}^{z}k(n)/n\right)
\]
\[
\hat{\gamma}_{n}^{m}:=M_{n}^{(1)}(k(n))+\hat{\gamma}_{n}^{-},\quad\hat{\sigma}_{n}:=X_{n-k(n)+1,n}M_{n}^{(1)}(k(n))(1-\hat{\gamma}_{n}^{-}),
\]
\[
\hat{\gamma}_{n}^{-}:=1-\frac{1}{2}\left(1-(M_{n}^{(1)}(k(n)))^{2}/M_{n}^{(2)}(k(n))\right)^{-1}
\]
and
\[
M_{n}^{(j)}(k):=\frac{1}{k-1}\sum_{i=1}^{k-1}\biggl(\log\frac{X_{n-i+1,n}}{X_{n-k+1,n}}\biggr)^{j}.
\]

This estimator is applicable to all $\gamma\in\mathbb{R}$, it is
accurate in comparison to other well-known estimators, and its bias
is small; see \emph{e.g.} \citet{Laurens  boek} (3.7.1). 

For each distribution function considered and each $n$ in $\{{\scriptstyle 2^{5}},{\scriptstyle 2^{6}},.....,{\scriptstyle 2^{16}}\}$,
1000 random samples were generated. The estimators were applied with
for each $n$, $k_{2}(n)$ and $k(n)$ chosen to minimise the empirical
mean square of $\log(\hat{\nu}_{n}(2\log n)+1)$ and of $\log(\hat{\nu}_{n}^{m}(2\log n)+1)$,
respectively. The reason for using say, $\log(\hat{\nu}_{n}+1)$ instead
of $\hat{\nu}_{n}$ is that its empirical distributions tend to be
more symmetrical with fewer outliers; note that $\log(\hat{\nu}_{n}+1)$
can replace $\hat{\nu}_{n}$ in Theorems \ref{thm:Pickands-like_random}
and \ref{thm:logGW_normal} and Corollary \ref{cor:normality}. The
reason for comparing estimates at the optimal $k_{2}(n)$ and $k(n)$
for each $n$ is to avoid biasing the comparison in favour of either
estimator. In addition, all quantile estimates were constrained from
below by the sample maxima.

Both the normal and lognormal distribution function satisfy $U\in ERV_{\{0\}}$\emph{.}
For the normal distribution, $q(y)\sim\sqrt{2y}$ as $y\rightarrow\infty$,
so $q\in ERV_{\{\nicefrac{1}{2}\}}$ and by Lemma \ref{lem:f-logf}(a)
in Subsection \ref{sub:Lemma}, $\log q\in ERV_{\{0\}}$; moreover,
$q'/q\in ERV_{\left\{ -1\right\} }$. Similarly, for the lognormal
distribution, it can be shown that $\log q\in ERV_{\{\nicefrac{1}{2}\}}$
and $q'/q\in ERV_{\left\{ -1/2\right\} }$. Therefore, Theorems \ref{thm:Pickands-like_random}
and \ref{thm:logGW_normal} apply to both distribution functions.
However, by Proposition \ref{pro:spec-lim}(b), neither satisfies
(\ref{eq:GP_conv}), and the lognormal does not even satisfy (\ref{eq:GP_conv_log}),
so we would not expect good performance of a GP-based quantile estimator
for $U(n^{2})$. 

\begin{figure}
\protect\caption{High quantile estimates for probabilities of exceedance of $n^{-2}$
on simulated independent standard lognormal samples based on GP (top)
and log-GW (bottom) based estimators as functions of $n$ (see text).
Diamonds/vertical bars: median of estimates (black) with 90\% intervals.
Left (a): quantile estimates, with target quantiles $U(n^{2})$ (dashed)
and approximate thresholds $U(n/k(n))$ and $U(n/k_{0}(n))$ (squares).
Centre (b): parameter estimates $\hat{\gamma}_{n}^{m}$ (top) and
$\hat{\theta}_{n}$ (bottom), dashed lines indicating the indices
$\gamma$ and $\theta$. Right (c): errors $\hat{\nu}_{n}^{m}$ (top)
and $\hat{\nu}_{n}$ (bottom). For log-GW only: quantile approximations
(-) and asymptotic 90\% interval bounds (-.). }
\label{Lognormal_many}

\begin{centering}
\includegraphics[width=3.8cm,height=3.8cm]{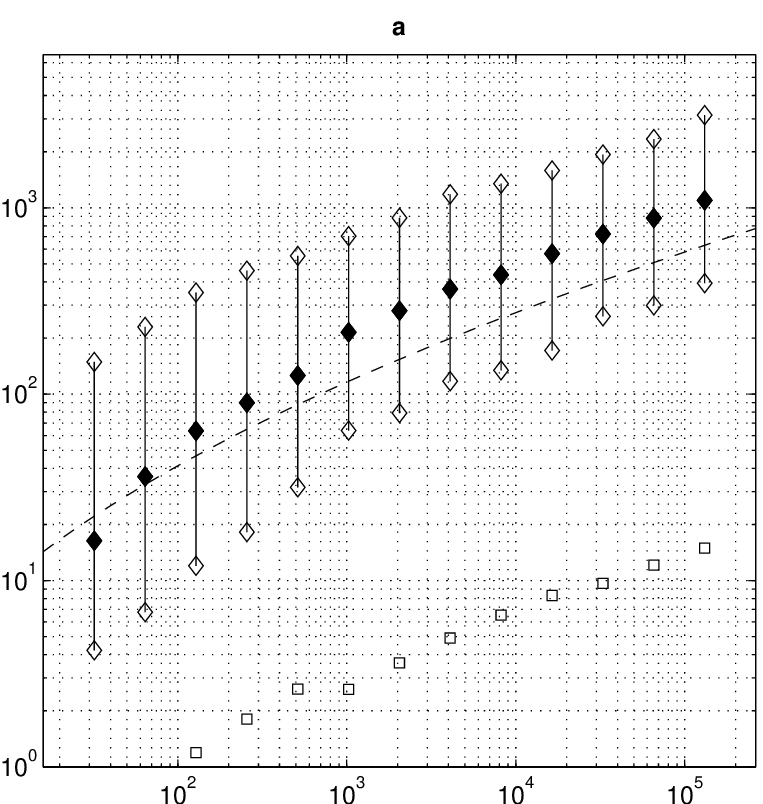}\includegraphics[width=3.8cm,height=3.8cm]{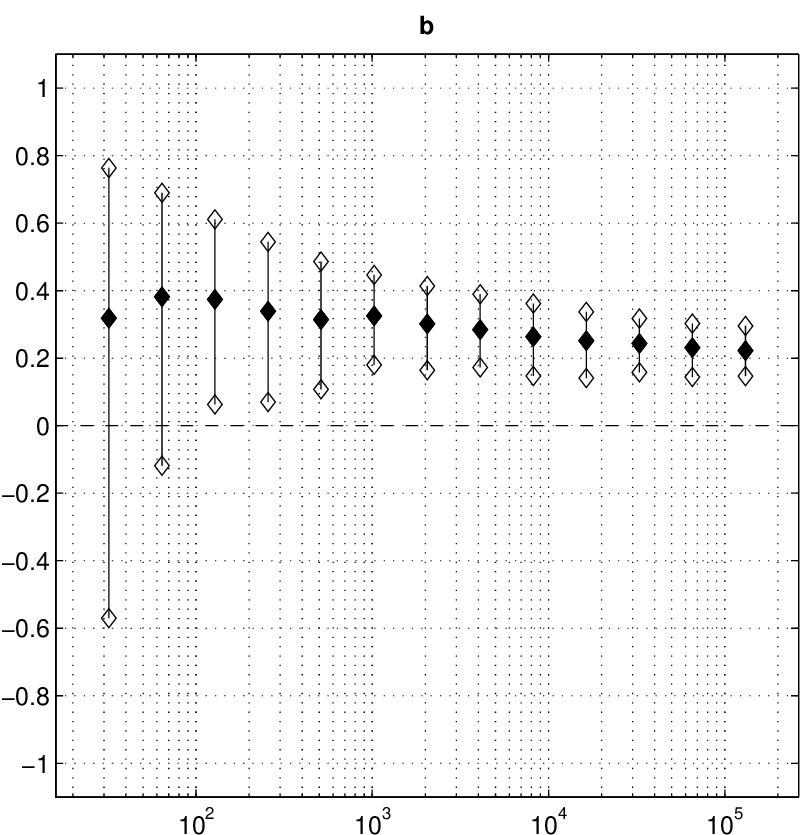}\includegraphics[width=3.8cm,height=3.8cm]{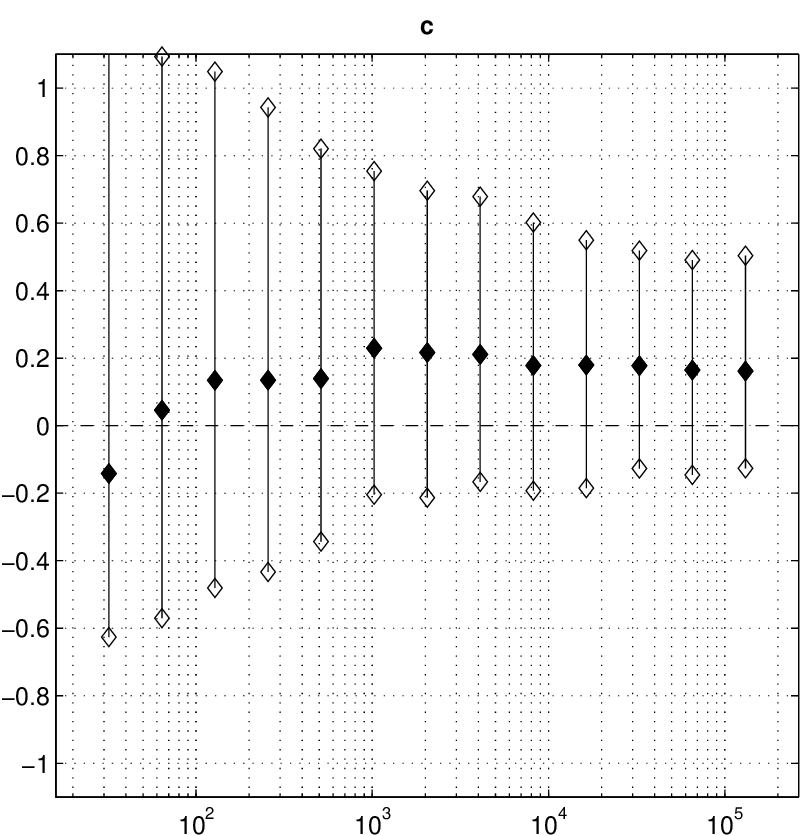}
\par\end{centering}

\centering{}\includegraphics[width=3.8cm,height=3.8cm]{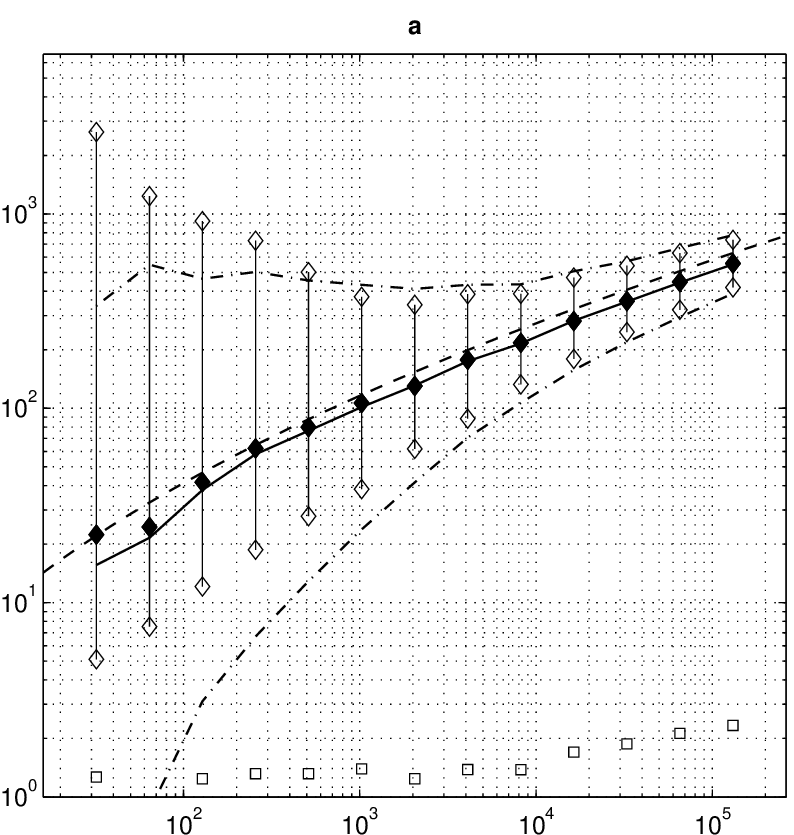}\includegraphics[width=3.8cm,height=3.8cm]{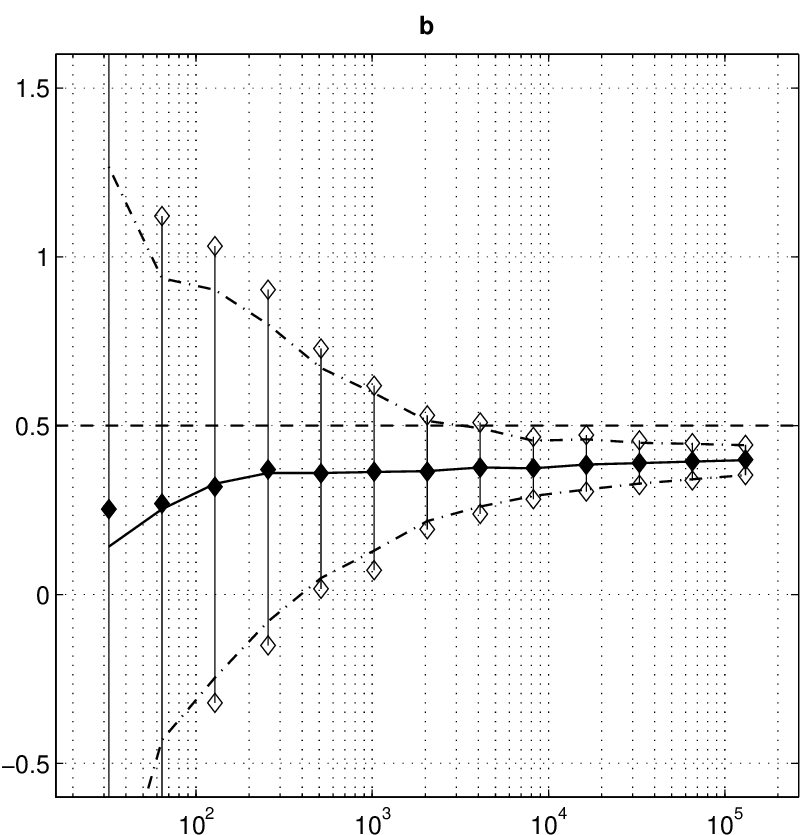}\includegraphics[width=3.8cm,height=3.8cm]{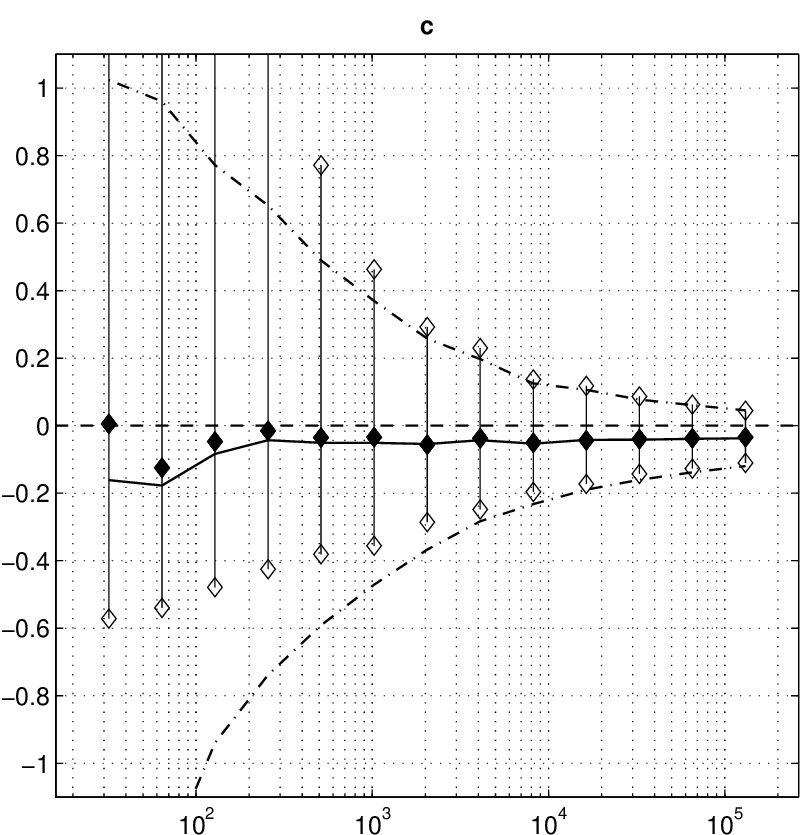}
\end{figure}
\begin{figure}
\protect\caption{As Figure \ref{Lognormal_many}, but for the standard normal distribution
instead of the lognormal.}
\label{Normal_many}

\begin{centering}
\includegraphics[width=3.8cm,height=3.8cm]{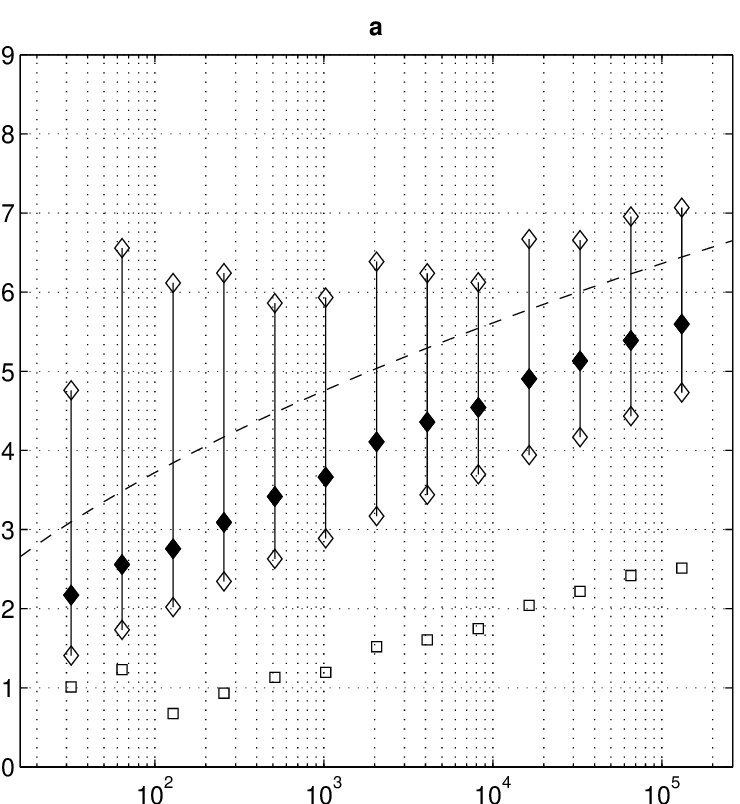}\includegraphics[width=3.8cm,height=3.8cm]{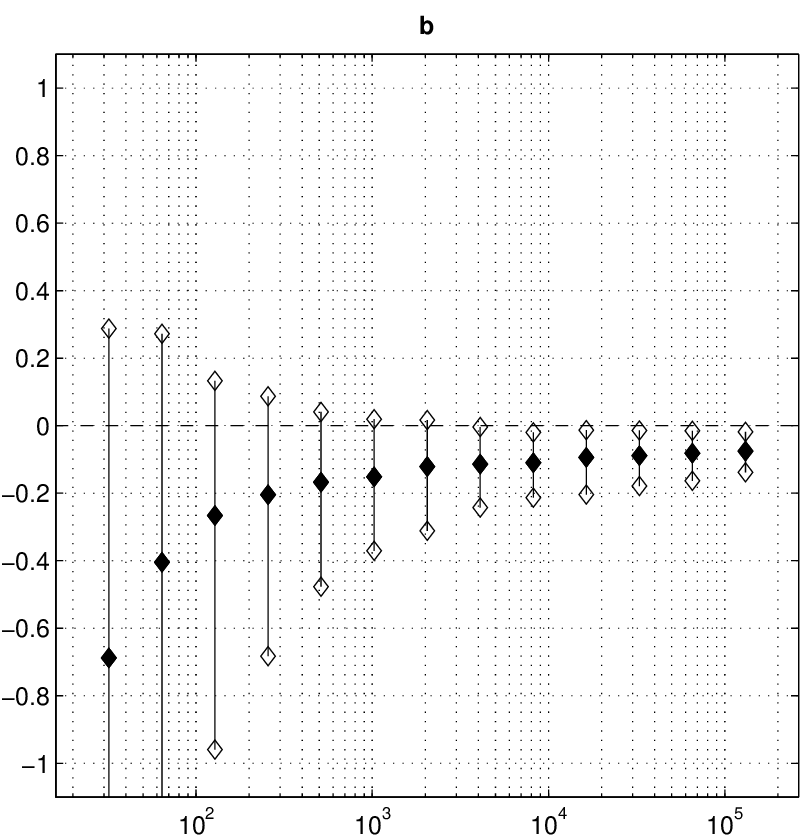}\includegraphics[width=3.8cm,height=3.8cm]{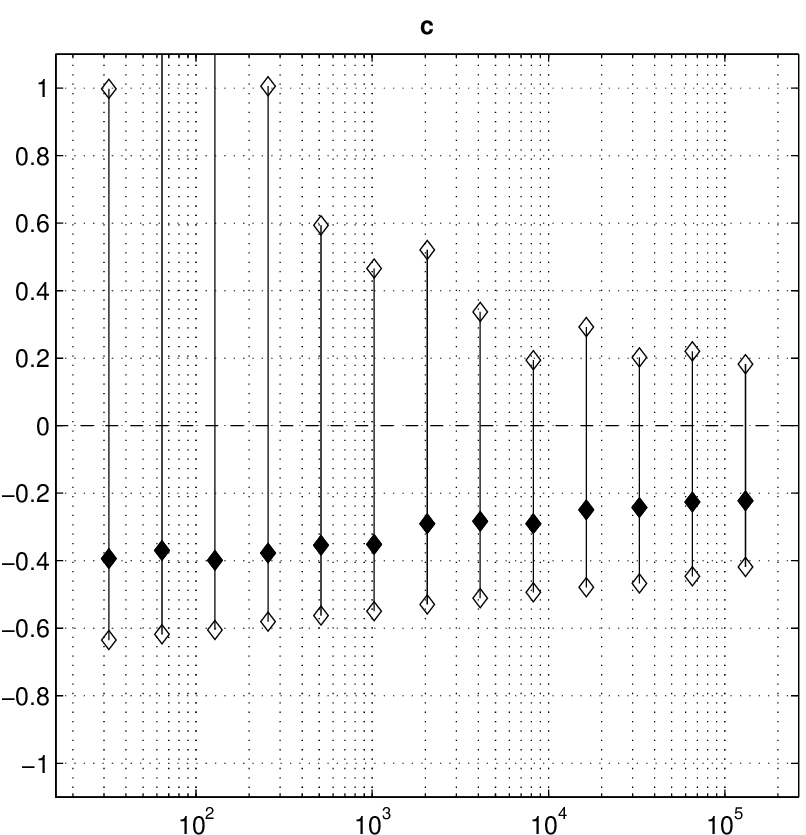}
\par\end{centering}

\centering{}\includegraphics[width=3.8cm,height=3.8cm]{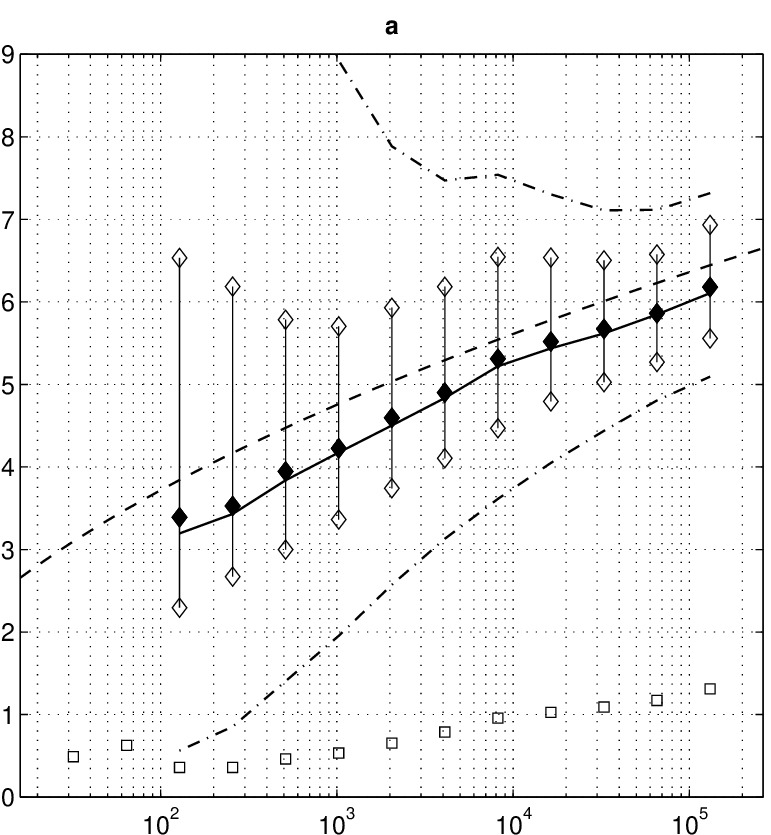}\includegraphics[width=3.8cm,height=3.8cm]{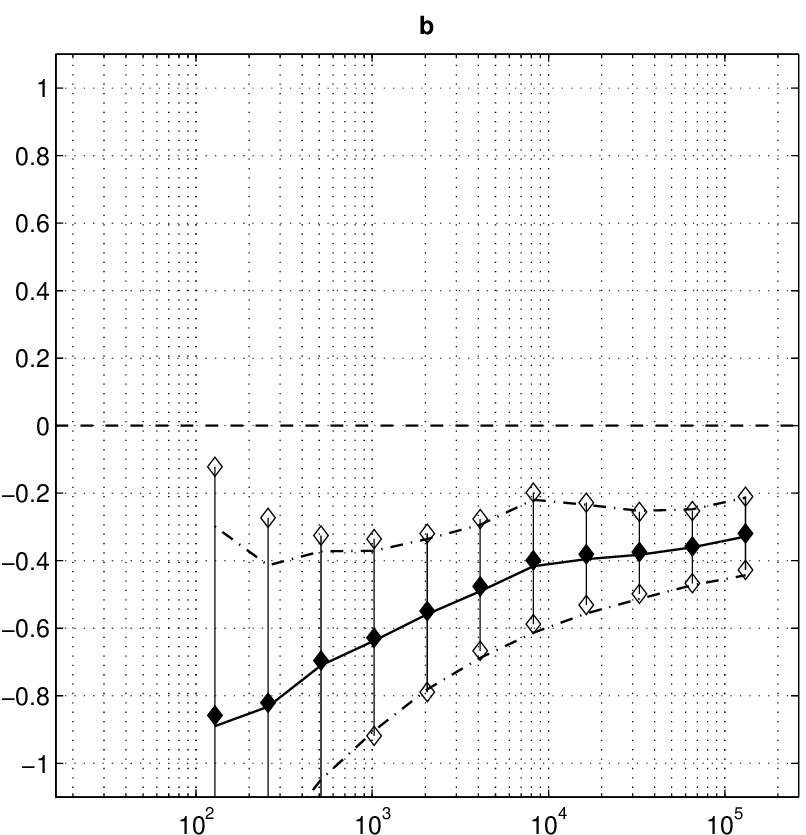}\includegraphics[width=3.8cm,height=3.8cm]{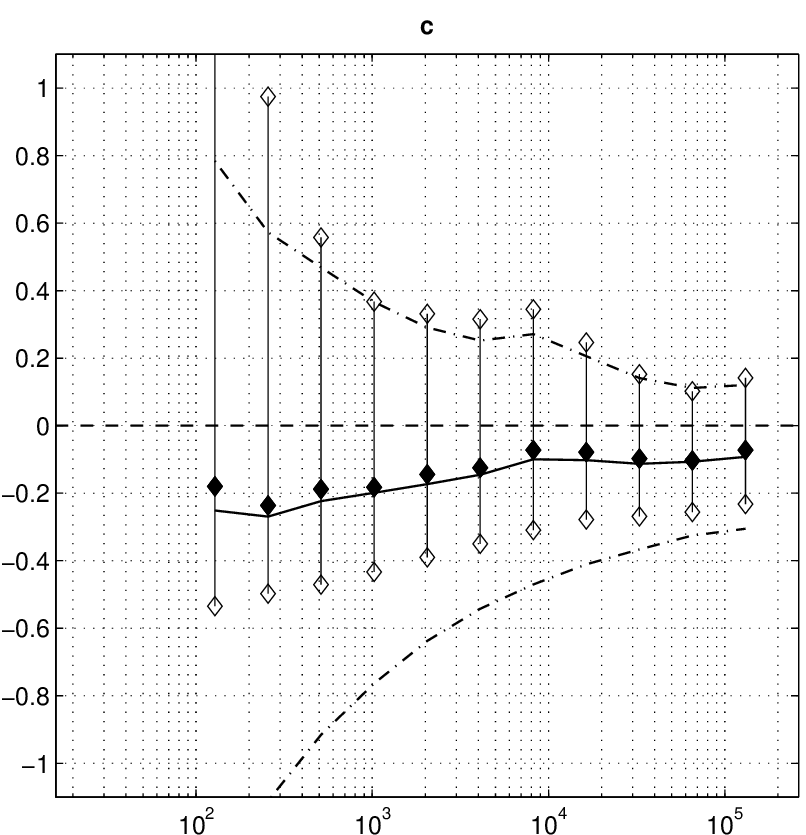}
\end{figure}
\begin{figure}
\protect\caption{As Figure \ref{Lognormal_many}, but for the Pareto-like distribution
($U(t)=t(1+2(\log t){}^{2})-1$) instead of the lognormal.}
\label{Frechetlike}

\begin{centering}
\includegraphics[width=3.8cm,height=3.8cm]{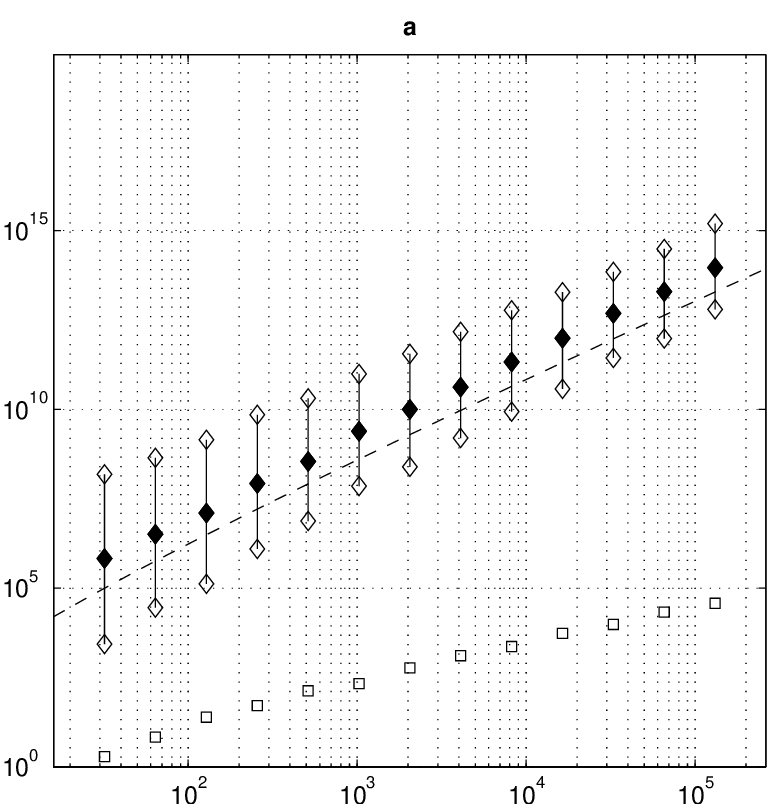}\includegraphics[width=3.8cm,height=3.8cm]{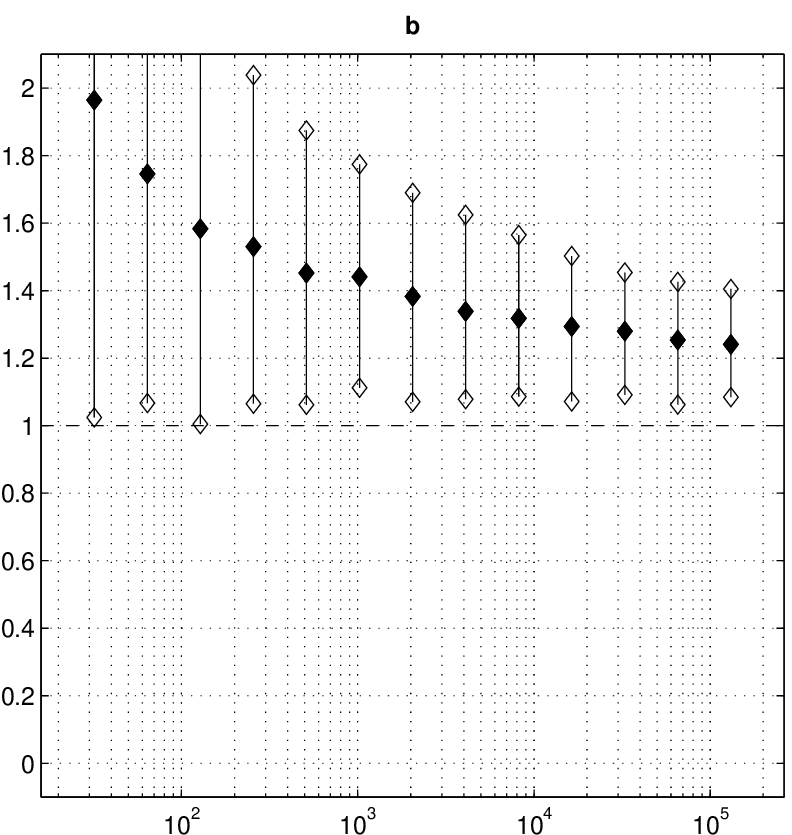}\includegraphics[width=3.8cm,height=3.8cm]{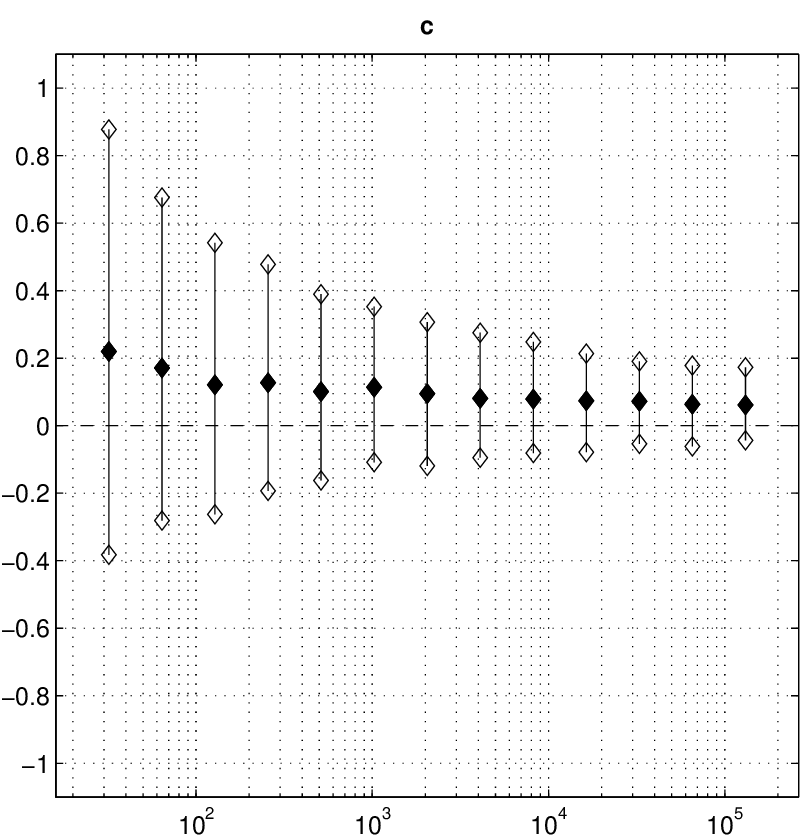}
\par\end{centering}

\centering{}\includegraphics[width=3.8cm,height=3.8cm]{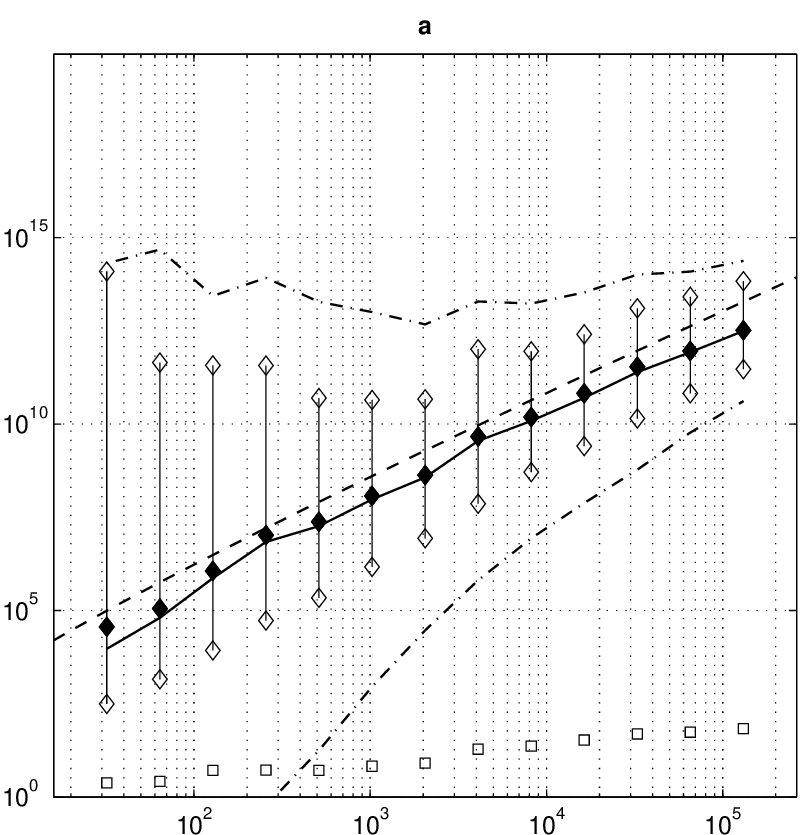}\includegraphics[width=3.8cm,height=3.8cm]{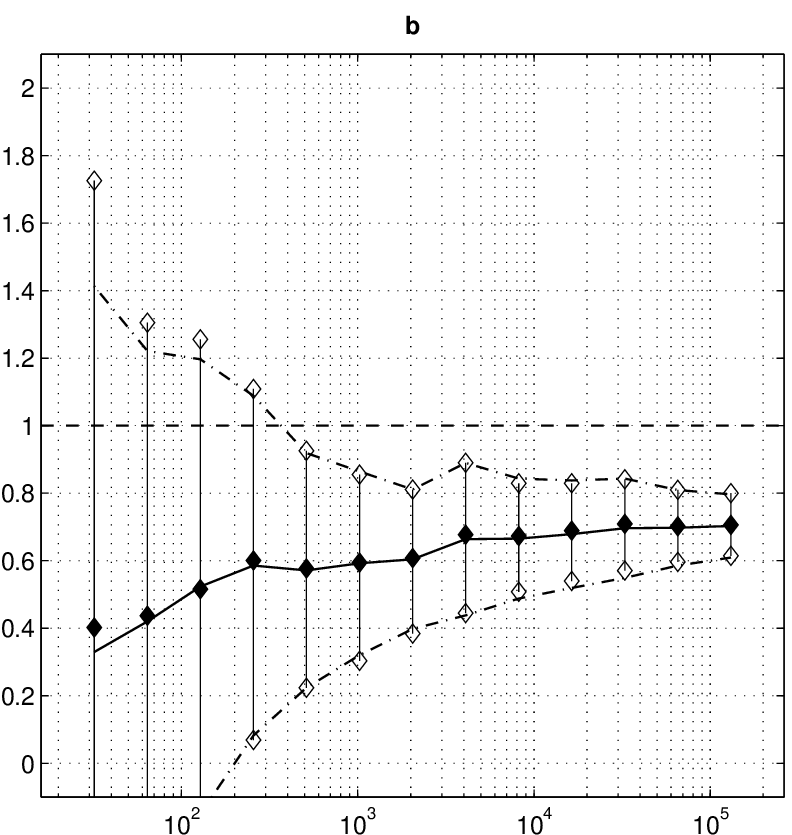}\includegraphics[width=3.8cm,height=3.8cm]{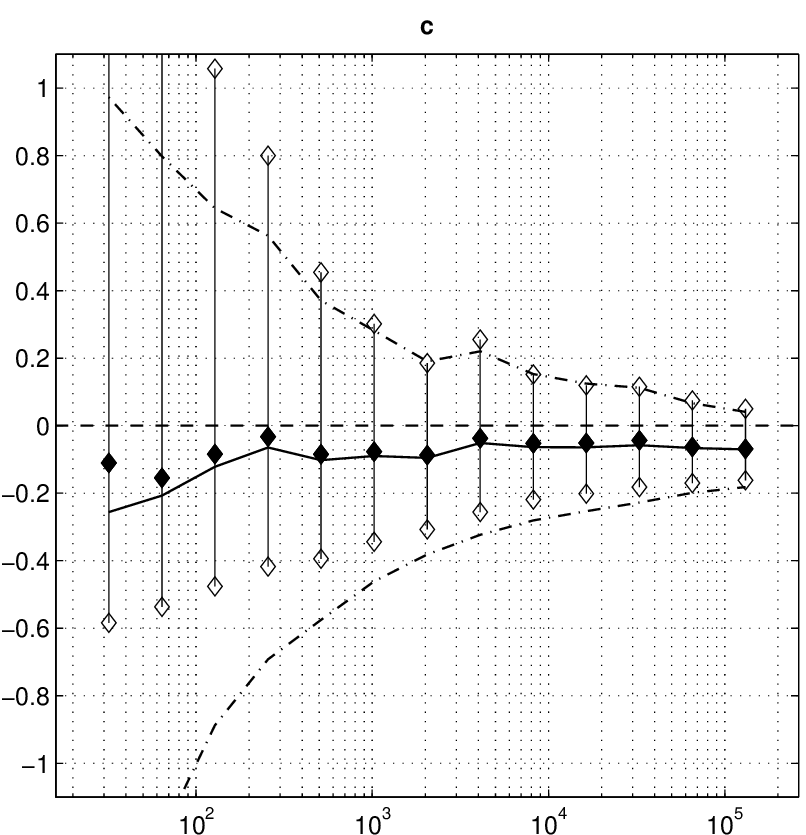}
\end{figure}
\begin{figure}
\protect\caption{As Figure \ref{Lognormal_many}, but for the Burr(1,$\frac{1}{4}$,4)
distribution (see main text) instead of the lognormal.}
\label{Burr}

\begin{centering}
\includegraphics[width=3.8cm,height=3.8cm]{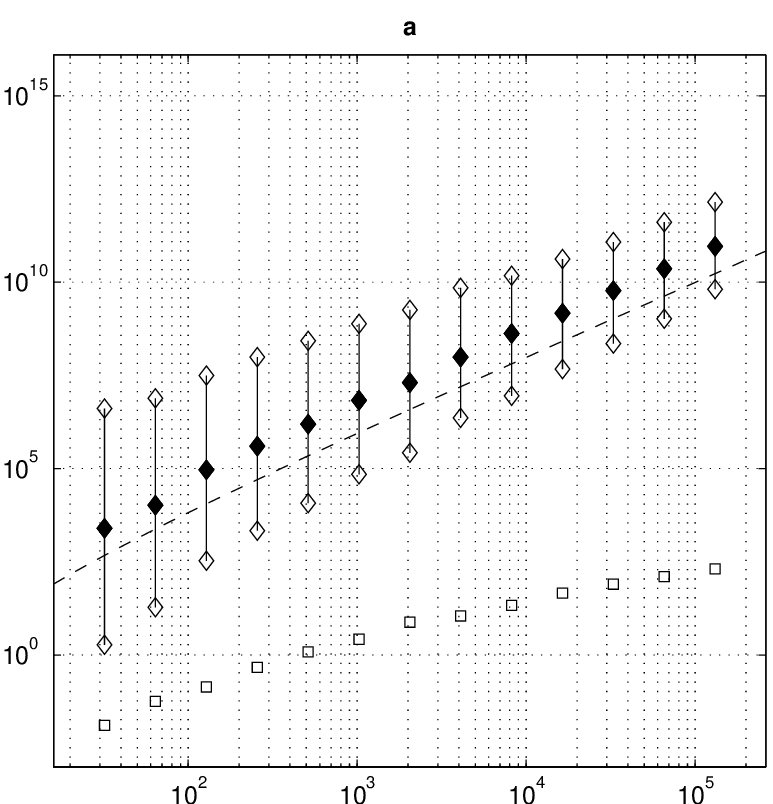}\includegraphics[width=3.8cm,height=3.8cm]{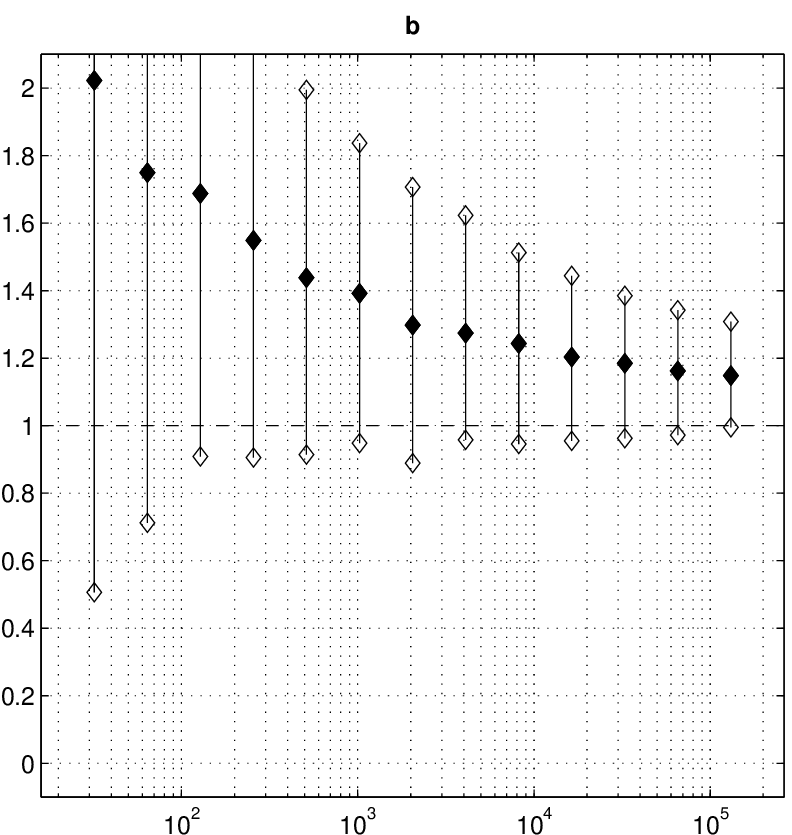}\includegraphics[width=3.8cm,height=3.8cm]{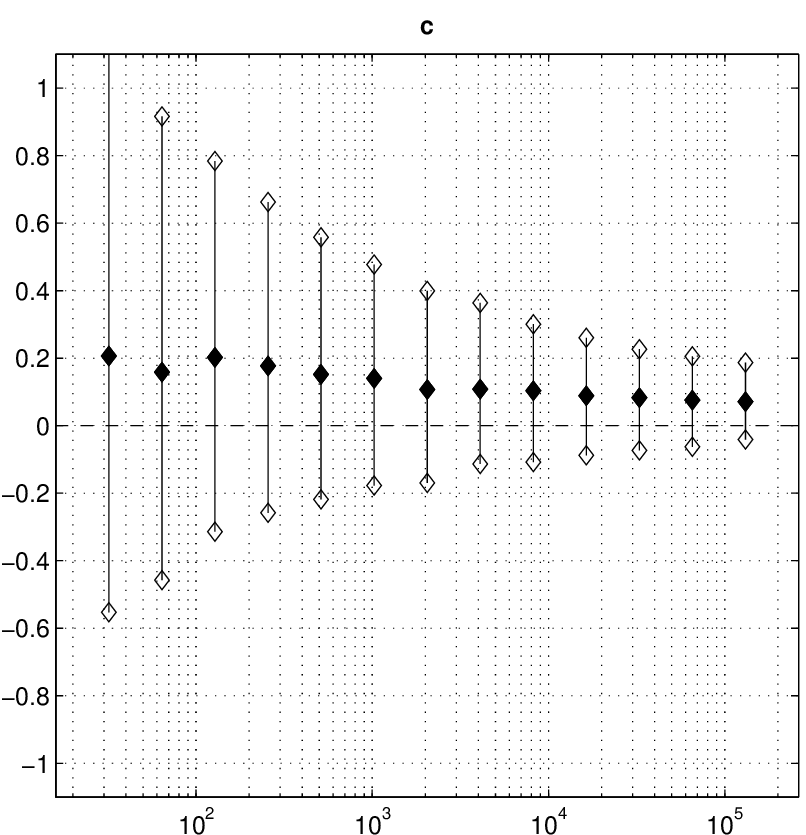}
\par\end{centering}

\centering{}\includegraphics[width=3.8cm,height=3.8cm]{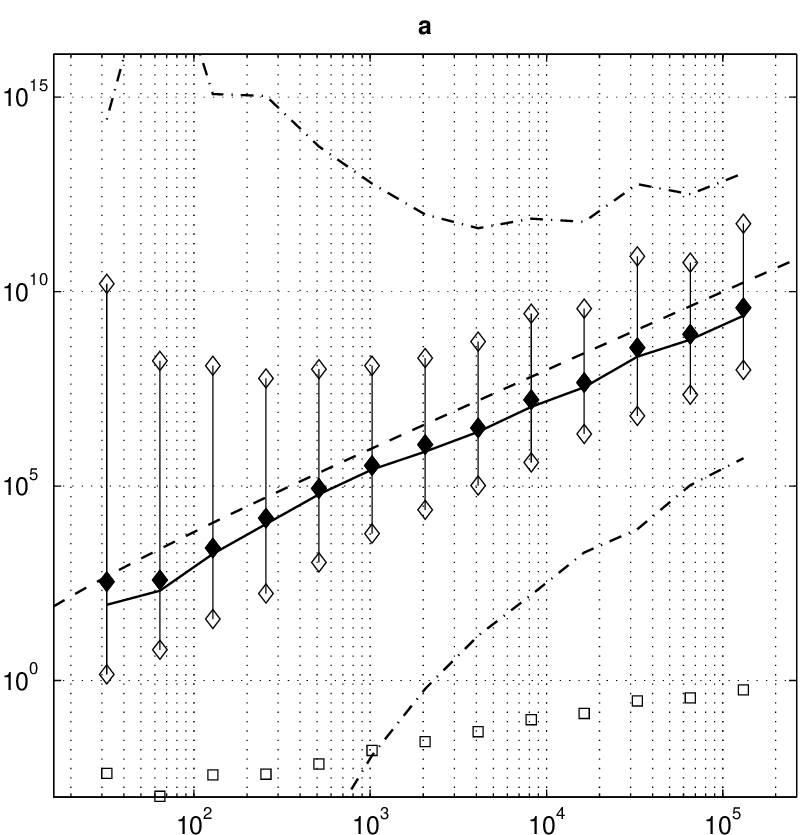}\includegraphics[width=3.8cm,height=3.8cm]{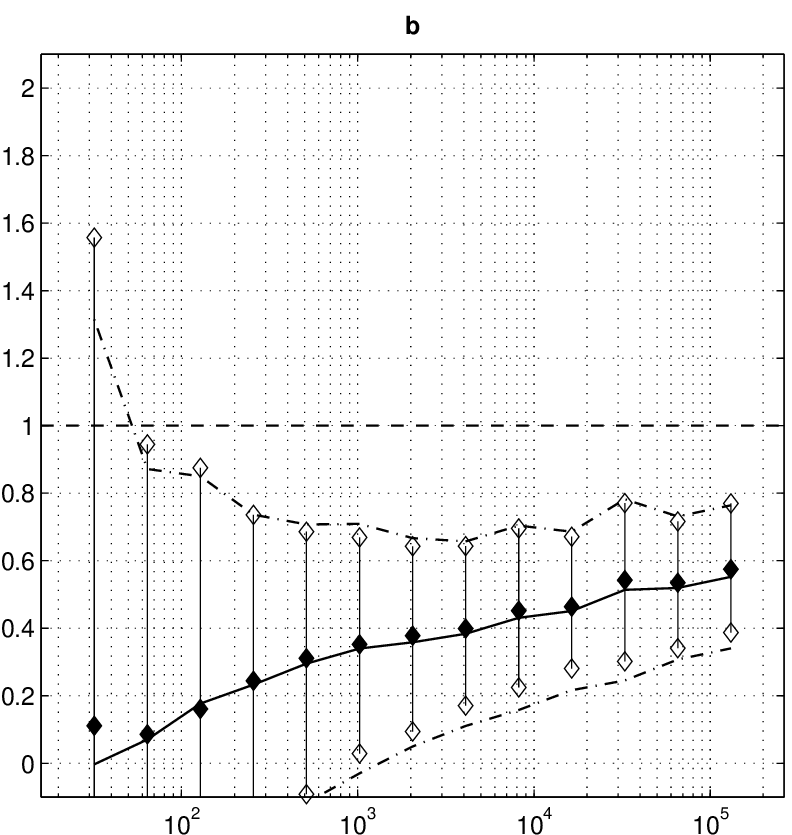}\includegraphics[width=3.8cm,height=3.8cm]{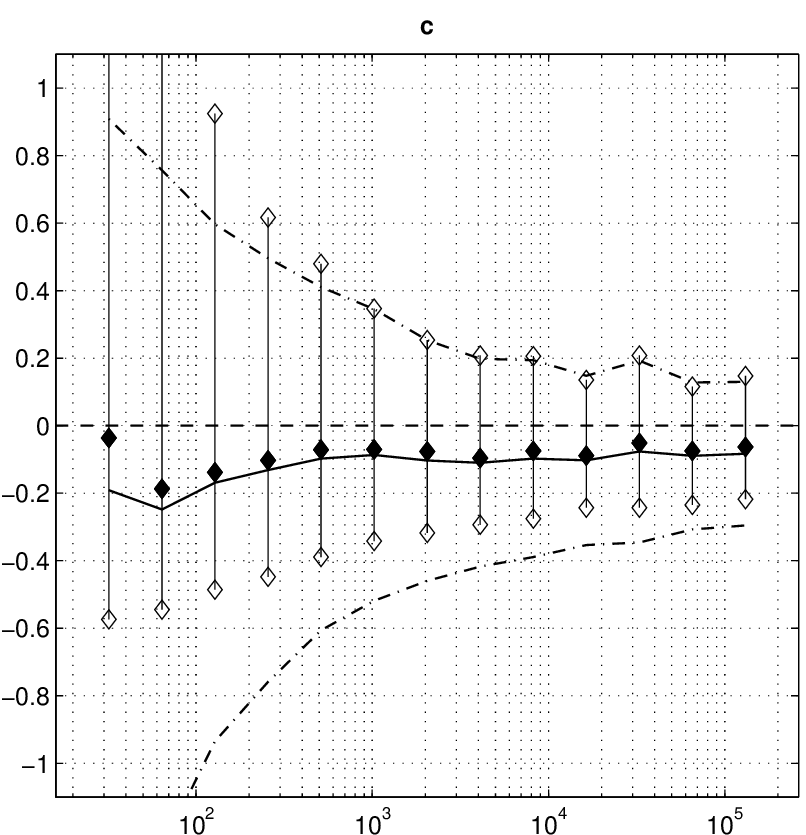}
\end{figure}

Figure \ref{Lognormal_many} shows the results for the lognormal distribution
with the GP-based estimator in the top row, and with the log-GW-based
estimator in the bottom row. The leftmost column (a) shows the medians
and the empirical 90\%-intervals (between the 5\% and 95\% percentiles)
of the quantile estimates; the width of an empirical 90\%-interval
will be referred to as ``spread''. The quantiles $U(n^{2})$ to
be estimated are indicated by a dashed curve. Approximate thresholds
$U(n/k(n))$ and $U(n/k_{0}(n))$ are indicated by open squares. The
middle column (b) shows the parameter estimates $\hat{\gamma}_{n}^{m}$
(top) and $\hat{\theta}_{n}$ (bottom), with the dashed lines indicating
the tail indices $\gamma$ and $\theta$ for the distribution function
considered. The rightmost column (c) displays the probability-based
errors $\hat{\nu}_{n}^{m}$ (top) and $\hat{\nu}_{n}$ (bottom). For
the log-GW-based estimator, also deterministic approximations $\tilde{q}_{y_{n}}(2\log n)$
with with $\tilde{\theta}=a_{\iota}$ and $\tilde{g}=g_{\iota}$ in
(\ref{eq:logGW_tailmodel}) and asymptotic 90\% intervals based on
(\ref{eq:rho_asnormal})-(\ref{eq:q_normal}) are displayed. The latter
are not confidence intervals, but are shown for comparison against
the empirical 5\% and 95\% percentiles and medians of the (biased)
estimates in order to verify how good the approximations provided
by (\ref{eq:rho_asnormal})-(\ref{eq:q_normal}) are. 

The top row of Figure \ref{Lognormal_many} shows the GP-based estimates
of $\log U(n^{2})$ apparently settling at a fixed distance upward
from the exact values, and no convergence of $\hat{\nu}_{n}^{m}$.
The parameter estimates $\hat{\gamma}_{n}^{m}$ appear to converge
slowly. In the bottom row, the log-GW-based estimator is seen to perform
well, with bias rapidly vanishing. Also, the spreads in $\hat{q}_{n}$
and $\hat{\nu}_{n}$ drop much more rapidly with increasing $n$ than
for $\hat{q}_{n}^{m}$ and $\hat{\nu}_{n}^{m}$. 

Figure \ref{Normal_many} for the normal distribution displays a similar
pattern as Figure \ref{Lognormal_many}, but with some differences.
The GP-based estimator now underestimates the very high quantiles,
even though the parameter estimator $\hat{\gamma}_{n}^{m}$ converges
rapidly. This is the only case in which the sample maximum as lower
bound to the quantile estimate became effective. The log-GW-based
quantile estimator is performing much better in this case, although
convergence is not as rapid as with lognormal data. Based on these
results alone, it is not clear whether the bias in $\hat{\nu}_{n}$
converges to zero; deterministic computations (not shown) for $n$
up to ${\scriptstyle 2^{60}}$ with prescribed $k_{2}(n)={\scriptstyle \left\lfloor n^{1/4}\right\rfloor }$
show that it vanishes slowly. For $\hat{q}_{n}(2\log n)-q(2\log n)$,
a small nonzero bias eventually remains, but the error relative to
$q(2\log n)$ vanishes, and therefore also the error relative to $q(2\log n)-q(y_{n})$. 

Since the favourable results of the log-GW-based estimator on lognormal
data would translate directly to equivalent results with an analogous
GW-based estimator on normal data, the latter would do better on the
normal data than the log-GW-based quantile estimator in Figure \ref{Normal_many}.
This indicates that in some cases, the speed of convergence may be
increased by replacing the latter by a GW-based estimator. 

The next two examples concern heavy-tailed distribution functions
with classical Pareto tail limits $U\in ERV_{(0,\infty)}$. By Proposition
\ref{pro:spec-lim}(a), $\log q\in ERV_{\{1\}}$. Figure \ref{Frechetlike}
shows results obtained for the distribution function satisfying $U(t)=t(1+2(\log t){}^{2})-1$.
Concerning bias, both estimators perform rather well as expected.
For small $n$, the log-GP-based estimator has a much smaller spread
than the log-GW estimator; for large $n$, the spreads are similar.
Given that the log-GW based estimator is based on only three order
statistics, a large small-sample spread is not surprising. Indeed,
replacing the moment estimator by Pickands' estimator for $\gamma$
(\citet{PickandsIII}) based on three order statistics, the spread
becomes larger than for the log-GW estimator (result not shown). 

Finally, Figure \ref{Burr} shows results for the Burr(1,$\frac{1}{4}$,4)
distribution with $U(t)={\scriptstyle (t^{1/4}-1)^{4}}$, which also
satisfies $U\in ERV_{\{1\}}$. Unlike the previous examples, $U$
has a negative second-order index (see \citet{Laurens  boek}) in
this case, so eventually, convergence toward the GP limit should be
rapid. As in the previous example, the GP-based estimator performs
rather well; the log-GW-based estimator performs similarly but with
somewhat larger spread (which is again smaller than obtained when
using Pickands' estimator for the GP-based estimation). 

In all figures, the threshold values $U(n/k_{0}(n))$ and $U(n/k(n))$
corresponding to the numbers $k_{0}(n)$ and $k(n)$ of upper order
statistics used by the estimators are rather different for the two
estimators. For the log-GW estimator with $\iota=2$, $k_{0}(n)$
must be at least $n^{3/4}$ irrespective of $k_{2}(n)$ (see (\ref{eq:k_j})),
so there is little room for adjustment in order to optimise performance.
However, with the GP estimator, $k(n)$ can be reduced considerably
to reduce bias if needed. For the lognormal and normal distributions
in Figures \ref{Lognormal_many} and \ref{Normal_many}, this is seen
to lead to a large spread.

The results tentatively confirm the expectations. For distribution
functions in the classical Pareto ($\gamma>0$) domain of attraction
(satisfying $\log q\in ERV_{\{1\}}$), log-GW en GP-based estimators
seem to perform similarly. However, in the classical domain of attraction
of the exponential ($\gamma=0$), log-GW may offer advantages. 

For the log-GW-based estimator, the asymptotic 90\%-intervals for
for $\hat{\theta}_{n}$ and $\hat{\nu}_{n}$ based on (\ref{eq:rho_asnormal})
and (\ref{eq:nu_normal}) provide good approximations to the empirical
90\%-intervals. For $\hat{q}_{n}$, the asymptotic 90\% intervals
based on (\ref{eq:q_normal}) are in some cases much too wide.

\section{\label{sec:Discussion}Discussion}

The log-GW tail limit $\log q\in ERV$ is a weak assumption of the
same nature as the classical regularity assumption $U\in ERV$ corresponding
to the GP tail limit, but specifically aimed toward approximation
and estimation of very high quantiles for probabilities in the range
(\ref{eq:p_n}). Proposition \ref{pro:spec-lim} indicates that if
a GP tail limit applies, then log-GW-based approximation may provide
benefits if $\gamma=0$. If $\gamma>0$, then approximation using
the GP tail should already be adequate. 

Further analysis confirms this: if $U\in ERV_{\{0\}}$ and $\log q\in ERV$,
then $\log q\in ERV_{(-\infty,1]}$ (see Theorem \ref{thm:GW_in_gamma=00003D0}(b)),
offering a continuum of tail shapes for approximation of quantiles
where the GP limit with $\gamma=0$ offers only one, the exponential
tail. 

Suppose that $U\in ERV_{\{0\}}$. Then any assumption ensuring convergence
of GP-based quantile approximations with $\gamma=0$ as in (\ref{eq:GP_conv})
implies $q\in RV_{\{1\}}$, so $\log q\in ERV_{\{0\}}(1)$ (see Proposition
\ref{pro:spec-lim}(b)); therefore, it excludes all other distribution
functions satisfying a log-GW tail limit and a GP tail limit with
$\gamma=0$, such as Weibull-like distributions (\emph{e.g.} the normal
distribution), distribution functions of exponents of Weibull-like
distributed random variables (\emph{e.g.} the lognormal distribution),
light tails with $q(\infty)$ still infinite such as $F=1-\exp(-\exp\textrm{I\textrm{d}})$,
distribution functions with finite $q(\infty)$ such as $F=1-\exp((q(\infty)-\textrm{Id})^{1/\theta}$)
with $\theta<0$, just to mention a few which correspond to log-GW
or GW limits or are close to such limits. 

As an example, consider the following seemingly innocent rate assumption
for (\ref{eq:U exc}): 
\begin{equation}
\lim_{t\rightarrow\infty}\biggl(\frac{U(t\lambda)-U(t)}{w(t)}-h_{\gamma}(\lambda)\biggr)\log t=0\quad\forall\lambda\geq1\label{eq:log_rate}
\end{equation}
with $\gamma=0$. It implies $q\in ERV_{\{1\}}$ (see Subsection \ref{sub:Proof-lograte}),
and thus by Lemma \ref{lem:f-logf}(a) in Subsection \ref{sub:Lemma},
$\log q\in ERV_{\{0\}}(1)$. Therefore, \emph{in the present context},
(\ref{eq:log_rate}) is actually quite restrictive. 

The Pareto domain of attraction with $\gamma>0$ is in the domain
of attraction of the log-GW tail limit $\log q\in ERV_{\{1\}}$ (Proposition
\ref{pro:spec-lim}(a)), so all results obtained for the latter also
apply to the former. Therefore, one might expect that if log-GW based
quantile approximation and estimation can not offer improvement if
$\gamma>0$, it may not do much harm either. 

This is tentatively confirmed by the results of the simulations in
Section \ref{sec:simulations}, which indicate that log-GW-based quantile
estimation may have merits within the $\gamma=0$ subdomain of attraction
of the GP limit and performs similarly to GP-based quantile estimation
in the $\gamma>0$ subdomain. However, it would be premature to draw
conclusions from only these few examples.

For the log-GW based estimator $\hat{q}_{n}$, the log-GW limit is
sufficient for consistency. To establish asymptotic normality, a relatively
high rate of convergence (\ref{eq:conv_rate}) to the log-GW limit
needed to be assumed. As shown in Section \ref{sec:estimators}, this
is a consequence of the particular formulation of this estimator.
Therefore, there is a need for alternative estimators which allow
the rate condition (\ref{eq:conv_rate}) to be relaxed. Based on the
simulation results, there appears to be a need for improved accuracy
with small sample sizes as well. It is suggested in Section \ref{sec:estimators}
that bias correction based on estimation of a higher-order ERV model
could be useful in log-GW-based estimators in order to obtain asymptotic
normality while avoiding slow decay of variability with increasing
$n$.

A limitation of log-GW approximation and estimation is that the notions
of convergence in (\ref{eq:conv_ql_rr_log}) and (\ref{eq:conv_logqlhat})
may be weak and cannot be replaced by (\ref{eq:conv_ql_rr}) and (\ref{eq:conv_qlhat})
for tails heavier than a typical Weibull tail unless additional assumptions
apply. The probability-based errors (\ref{eq:conv_nul}) and (\ref{eq:conv_nuhat})
are based on log-ratios of survival functions of a quantile and its
approximation or estimator. Although natural in view of the probability
range considered, a stronger notion of convergence,\emph{ e.g.} of
a ratio of survival functions, would be desirable for applications.
Stronger notions of convergence apply under the additional assumptions
for establishing asymptotic normality for the quantile estimator $\hat{q}_{n}$
in Corollary \ref{cor:normality}, notably the rate assumption (\ref{eq:conv_rate});
see Remark \ref{rem:strong_conv}. 

As a final remark, log-GW-based quantile approximation and estimation
for light tails with finite endpoints has only been marginally covered
here, so this case remains to be examined in more detail.

\section{Proofs and lemmas\label{sec:proofs}}

\subsection{\label{sub:Proof-of-spec-lim}Proof of Proposition \ref{pro:spec-lim}}

If $U\in ERV_{\{\gamma\}}$ for $\gamma>0$, then $U\in RV_{\{\gamma\}}$
so by the Potter bounds (\textit{e.g.} \citet{Bingham}, Theorem 1.5.6),
there is for every $\varepsilon\in(0,\gamma\wedge1)$ a $y_{\varepsilon}>0$
such that $y(\lambda-1)(\gamma-\varepsilon)-\varepsilon\leq\log q(y\lambda)-\log q(y)\leq y(\lambda-1)(\gamma+\varepsilon)+\varepsilon$
for all $y\geq y_{\varepsilon}$ and all $\lambda\geq1$. Therefore,
$\log q\in ERV_{\{1\}}(\textrm{Id}\cdot\gamma)$, so $\log q\in RV_{\{1\}}$.
Noting that $\gamma y\sim\log q(y)$ as $y\rightarrow\infty$ and
$\gamma U(t)/w(t)\rightarrow1$ as $t\rightarrow\infty$ (both due
to \citet{Laurens  boek}, Theorem B.2.2(1)), we obtain $\log\tilde{U}_{t}(t^{\lambda})=\log U(t)+\log(1+\frac{w(t)}{\gamma U(t)}(t^{(\lambda-1)\gamma}-1)))$
$=\log U(t)+(\lambda-1)\gamma\log t+o(1)\sim\lambda\log U(t)$ for
every $\lambda\geq1$, and since $\log q\in RV_{\{1\}}$, (\ref{eq:GP_conv_log})
follows, so (a) is proven.

If $\gamma=0$, (\ref{eq:GP_conv}) implies $\lim_{y\rightarrow\infty}(q(y\lambda)-q(y))/(w(\textrm{e}^{y})y)=\lambda-1\quad\forall\lambda>1$.
Therefore, $q\in ERV_{\{1\}}$ and by Lemma \ref{lem:f-logf}(a) in
Subsection \ref{sub:Lemma}, $\log q\in ERV_{\{0\}}(1)$, proving
(b).

\subsection{\label{sub:proof-GW_in_gamma=00003D0}Proof of Theorem \ref{thm:GW_in_gamma=00003D0}}

Suppose that $U\in ERV_{\{\gamma\}}$ with $\gamma>0$, then as in
Subsection \ref{sub:Proof-of-spec-lim}, there is for every $\varepsilon\in(0,\gamma\wedge1)$
a $y_{\varepsilon}>0$ such that $y(\lambda-1)(\gamma-\varepsilon)-\varepsilon\leq\log(q(y\lambda)/q(y))\leq y(\lambda-1)(\gamma+\varepsilon)+\varepsilon$
for all $y\geq y_{\varepsilon}$ and all $\lambda\geq1$, and therefore,
fixing $\iota>1$ and $\xi>\iota$, there is some $\varepsilon\in(0,\gamma(\xi-\iota)/(\xi+\iota-2)\wedge1)$,
$\delta>0$ and $z_{\varepsilon}\geq y_{\varepsilon}$ such that 
\begin{equation}
\frac{q(y\xi)-q(y)}{q(y\iota)-q(y)}\geq\frac{e^{y(\xi-1)(\gamma-\varepsilon)}(1-\varepsilon)-1}{e^{y(\iota-1)(\gamma+\varepsilon)}(1+\varepsilon)-1}\geq\exp(\delta y)\quad\forall y\geq z_{\varepsilon}.\label{eq:limit_to_inf}
\end{equation}
However, since $q\in ERV$, the left-hand side of (\ref{eq:limit_to_inf})
must tend to $h_{\theta}(\xi)/h_{\theta}(\iota)<\infty$ for some
real $\theta$ as $y\rightarrow\infty$, so $\gamma$ cannot exceed
$0$. Assuming that $\gamma<0$, a similar argument leads to a similar
contradiction, completing the proof of (a). 

For (b), if $U\in ERV$ then by Lemma \ref{lem:f-logf}(a) in Subsection
\ref{sub:Lemma}, $\log U\in ERV$ so since $\log q\in ERV$, (a)
implies that $\log U\in ERV_{\{0\}}$. Since $U\in ERV$ and $\log q\in ERV$,
Proposition \ref{pro:spec-lim}(a) implies that either $U\in ERV_{(0,\infty)}$
and $\log q\in ERV_{\{1\}}$, or $U\in ERV_{(-\infty,0]}$. In the
latter case, since $\log U\in ERV_{\{0\}}$, Lemma \ref{lem:f-logf}(c)
implies that $U\in ERV_{\{0\}}\subset RV_{\{0\}}$. Therefore, by
the Potter bounds, $\log q(y)=o(y)$ as $y\rightarrow\infty$, so
again by the Potter bounds, $\log q$ cannot be in $RV_{(1,\infty)}=ERV_{(1,\infty)}$.

\subsection{\label{sub:proof-procentral}Proof of Proposition \ref{pro:central}}

Because $\tilde{\theta}(y)\rightarrow\theta$ and $\tilde{g}(y)\sim g(y)$
as $y\rightarrow\infty$, noting that by (\ref{eq:B}), $h_{\theta+o(1)}(\lambda)=\lambda^{o(1)}h_{\theta}(\lambda)$,
we obtain using the mean value theorem, 
\begin{equation}
\log\tilde{q}_{y}(y\lambda)=\log q(y)+g(y)h_{\theta}(\lambda)(1+o(1))\label{eq:q_tilde_expand}
\end{equation}
locally uniformly in $\lambda>0$. Since (\ref{eq:GW_explic}) also
holds locally uniformly in $\lambda>0$ (see \citet{Bingham}, Theorem
3.1.16), (\ref{eq:conv_ql_rr_log}) follows from (\ref{eq:q_tilde_expand}).
If in addition, (\ref{eq:g_evt_bnd}) holds, then by (\ref{eq:conv_ql_rr_log}),
$\log\tilde{q}_{y}(y\lambda)-\log q(y\lambda)=(\tilde{q}_{y}(y\lambda)/q(y\lambda)-1)(1+o(1))$
as $y\rightarrow\infty$ locally uniformly in $\lambda>0$, and (\ref{eq:conv_ql_rr})
follows. If $q\in ERV$, then (\ref{eq:g_evt_bnd}) follows from Lemma
\ref{lem:f-logf}(b) in Subsection \ref{sub:Lemma}.

\subsection{\label{sub:notes_Remark}Clarification of Remark \ref{rem:R1} }

Under condition (\ref{eq:g_evt_bnd}), $q(y)g(y)$ in (\ref{eq:conv_ql_rr})
can be replaced by $q(y\xi)-q(y)$ for any $\xi\in(0,\infty)\setminus\{1\}$:
for $\xi>1$, this follows from (\ref{eq:g_bound}), as $q(y\xi)/q(y)-1>\log q(y\xi)-\log q(y)$;
for $\xi\in(0,1)$, we find $\bigl|\frac{g(y)q(y)}{q(y\xi)-q(y)}\bigr|=\frac{g(y\xi)}{q(y)/q(y\xi)-1}O(1)=\frac{g(y\xi)}{\log q(y)-\log q(y\xi)}O(1)=O(1)$
as $y\rightarrow\infty$ by regular variation of $g$ and (\ref{eq:g_bound}).

If $q(\infty)<\infty$, then $q(y)g(y)$ in (\ref{eq:conv_ql_rr})
may be replaced by $q(\infty)-q(y\eta)$ for any $\eta>0$: taking
$\xi>1$, $\frac{q(y)g(y)}{q(\infty)-q(y\eta)}\leq\frac{q(\infty)g(y)}{q(y\eta\xi)-q(y\eta)}\sim\frac{g(y\eta)}{\log q(y\eta\xi)-\log q(y\eta)}\frac{g(y)}{g(y\eta)}=O(1)$
as $y\rightarrow\infty$ by (\ref{eq:g_bound}) and regular variation
of $g$.

\subsection{\label{sub:ProofProPHI}Proof of Theorem \ref{thm: PHI-omega-nu}}

From Proposition \ref{pro:central} and (\ref{eq:logq_in_PI_explic}),
as $y\rightarrow\infty$,
\begin{equation}
\log\tilde{q}_{y}(y\lambda)=\log q(y)+g(y)(h_{\theta}(\lambda)+o(1))\label{eq:O1}
\end{equation}
locally uniformly in $\lambda>0$. Let $\varLambda>1$ and $b\in(0,\varLambda^{-\bigl|\theta\bigr|}/\bigl|\theta\bigr|)$.
Applying the mean value theorem to $x\mapsto h_{\theta}^{-1}(h_{\theta}(\lambda)+x)=(\lambda^{\theta}+x\theta)^{1/\theta}$,
we find that for some $M>0$,
\[
\bigl|h_{\theta}^{-1}(h_{\theta}(\lambda)+x)-\lambda\bigr|\leq M\bigl|x\bigr|\qquad\forall\lambda\in[\varLambda^{-1},\varLambda],\: x\in[-b,b].
\]

Therefore, by (\ref{eq:O1}), $\log\tilde{q}_{y}(y\lambda)=\log q(y)+g(y)h_{\theta}(\lambda+o(1))$
uniformly in $\lambda\in[\varLambda^{-1},\varLambda]$, so using (\ref{eq:log-GW_surv_logform}),
\[
q^{-1}(\tilde{q}_{y}(y\lambda))=-\log\bigl(1-F\bigl(q(y)e^{g(y)h_{\theta}(\lambda+o(1))}\bigr)\bigr)=y(\lambda+o(1))
\]
uniformly in $\lambda\in[\varLambda^{-1},\varLambda]$. As $\lim_{z\rightarrow\infty}z^{-1}q^{-1}(q(z))=1$
by (\ref{eq:log-GW_surv_logform}), we obtain (\ref{eq:conv_nul}).

\subsection{\label{sub:Proof-of-Pickands-like}Proof of Theorem \ref{thm:Pickands-like_random}}

Define $\hat{\iota}_{m}(n)$ for all $n\geq1$ and $m\in\{0,1,2\}$
by
\begin{equation}
\hat{\iota}_{m}(n):=y_{n}^{-1}q^{-1}(X_{n-k_{m}(n)+1,n})=-y_{n}^{-1}\log(1-F(X_{n-k_{m}(n)+1,n})).\label{eq:iota_def}
\end{equation}

To simplify notation, we will use
\begin{equation}
s:=\log q,\quad\hat{s}_{n}:=\log\hat{q}_{n},\quad\tilde{s}_{y}:=\log\tilde{q}_{y}.\label{eq:s_defs}
\end{equation}

Since $q(\infty)>1$, almost surely some $n_{0}\in\mathbb{N}$ exists
such that for all $n\geq n_{0}$, $X_{n-k_{0}(n)+1,n}>0$ and $\hat{q}_{n}(z)$
is defined. By Lemma \ref{lem:random_argument_converges} in Subsection
\ref{sub:Lemma}, (\ref{eq:X=00003Dq_yiota}) and (\ref{eq:iota_converges2})
hold for $\hat{\iota}_{0}(n)$, $\hat{\iota}_{1}(n)$ and $\hat{\iota}_{2}(n)$
defined by (\ref{eq:iota_def}). Therefore, by (\ref{eq:def_ahat})
and (\ref{eq:X=00003Dq_yiota}), using (\ref{eq:s_defs}),
\[
\hat{\theta}_{n}=\frac{1}{\log\iota}\log\frac{s(y_{n}\hat{\iota}_{2}(n))-s(y_{n}\hat{\iota}_{1}(n))}{s(y_{n}\hat{\iota}_{1}(n))-s(y_{n}\hat{\iota}_{0}(n))}\quad\forall n\geq n_{0}\quad\textrm{a.s}.
\]
and as $s\in ERV_{\{\theta\}}(g)$ and therefore $g\in RV_{\{\theta\}}$,
by locally uniform convergence (see \citet{Bingham}, Theorems 3.1.16
and 1.5.2), and (\ref{eq:iota_converges2}), almost surely 
\begin{equation}
\hat{\theta}_{n}=\frac{1}{\log\iota}\left(\log\frac{h_{\theta}(\hat{\iota}_{2}(n)/\hat{\iota}_{1}(n))+o(1)}{h_{\theta}(\hat{\iota}_{1}(n)/\hat{\iota}_{0}(n))+o(1)}+\log\frac{g(y_{n}\hat{\iota}_{1}(n))}{g(y_{n}\hat{\iota}_{0}(n))}\right)\rightarrow\theta.\label{eq:rho_conv}
\end{equation}

Similarly, using (\ref{eq:rho_conv}), almost surely
\[
\frac{\hat{g}_{n}}{g(y_{n})}=\frac{s(y_{n}\hat{\iota}_{1}(n))-s(y_{n}\hat{\iota}_{0}(n)\eta)}{g(y_{n})h_{\hat{\theta}_{n}}(\iota)}=\frac{h_{\theta}(\hat{\iota}_{1}(n)/\hat{\iota}_{0}(n))+o(1)}{h_{\hat{\theta}_{n}}(\iota)}\left(\frac{g(y_{n}\hat{\iota}_{0}(n))}{g(y_{n})}\right)\rightarrow1,
\]
so (\ref{eq:conv_ahat}) is proven. Furthermore, in a similar manner,
almost surely
\begin{equation}
\frac{s(y_{n})-\log X_{n-k_{0}(n)+1,n}}{g(y_{n})}=\frac{s(y_{n})-s(y_{n}\hat{\iota}_{0}(n))}{g(y_{n})}=-h_{\theta}(\hat{\iota}_{0}(n))+o(1)\rightarrow0.\label{eq:s_diff}
\end{equation}

By (\ref{eq:conv_ahat}), almost surely $(\hat{g}_{n}/g(y_{n}))h_{\hat{\theta}_{n}}(\lambda)\rightarrow h_{\theta}(\lambda)$
locally uniformly in $\lambda>0$, so using (\ref{eq:s_diff}), $\hat{s}_{n}$
defined by (\ref{eq:s_defs}) and (\ref{eq:def_qhat}) satisfies
\begin{equation}
\sup_{\lambda\in[\Lambda^{-1},\Lambda]}\bigl|\hat{s}_{n}(y_{n}\lambda)-s(y_{n}\lambda)\bigr|/g(y_{n})\rightarrow0\qquad a.s.\qquad\forall\varLambda>1.\label{eq:s_hat_conv}
\end{equation}
 Using (\ref{eq:yn_over_logn}), we subsequently obtain (\ref{eq:conv_logqlhat}),
and (\ref{eq:conv_qlhat}) follows readily as in the proof of Proposition
\ref{pro:central}. From (\ref{eq:s_hat_conv}) and (\ref{eq:logq_in_PI_explic}),
almost surely
\begin{equation}
\hat{s}_{n}(y_{n}\lambda)=s(y_{n}\lambda)+o(g(y_{n}))=s(y_{n})+g(y_{n})(h_{\theta}(\lambda)+o(1))\label{eq:shat_appr}
\end{equation}
locally uniformly in $\lambda>0$. By mimicking the proof of Theorem
\ref{thm: PHI-omega-nu} in Subsection \ref{sub:ProofProPHI} with
$y_{n}$ replacing $y$ and $\hat{s}_{n}(y_{n}\lambda)$ replacing
$\tilde{s}_{y}(y\lambda)$, we obtain that for every $\varLambda>1$
almost surely, $\sup_{\lambda\in[\Lambda^{-1},\Lambda]}\bigl|\hat{\nu}_{n}(y_{n}\lambda)\bigr|\rightarrow0$;
using (\ref{eq:yn_over_logn}), (\ref{eq:conv_nuhat}) follows.

\subsection{\label{sub:Proof-of-Theorem-GWnormal}Proofs of Theorem \ref{thm:logGW_normal}
and Corollary \ref{cor:normality}}

Using the definitions (\ref{eq:s_defs}), by (\ref{eq:q_diff_cond}),
$s'\in RV_{\{\theta-1\}}$, so $s$ is a homeomorphism on some neighbourhood
of $\infty$. Therefore, without loss of generality, we can take $s$
increasing and continuous, so $\log X_{n-k_{m}(n)+1,n}=s(y_{n}\hat{\iota}_{m}(n))$
for all $n$ and $m\in\{0,1,2\}$. Furthermore, by integration, $s'\in RV_{\{\theta-1\}}$
implies
\begin{equation}
s(y\lambda)-s(y)=s'(y)yh_{\theta}(\lambda)(1+o(1))\label{eq:multipl_error}
\end{equation}
with $o(1)$ vanishing locally uniformly for $\lambda>0$ as $y\rightarrow\infty$%
\footnote{This implies $\log U\in ERV_{\{0\}}$, supplementing Theorem \ref{thm:GW_in_gamma=00003D0}. %
}. Therefore, with $\hat{\iota}$ as in (\ref{eq:iota_def}) and
\[
\mathcal{R}_{n}^{m}:=\frac{\log X_{n-k_{m}(n)+1,n}-s(y_{n}\iota^{m})}{s'(y_{n}\iota^{m})y_{n}\iota^{m}},
\]
using (\ref{eq:iota_converges2}) from Lemma \ref{lem:random_argument_converges}
in Subsection \ref{sub:Lemma}, almost surely
\begin{equation}
\mathcal{R}_{n}^{m}=\frac{s(y_{n}\hat{\iota}_{m}(n))-s(y_{n}\iota^{m})}{s'(y_{n}\iota^{m})y_{n}\iota^{m}}\sim h_{\theta}(\iota^{-m}\hat{\iota}_{m}(n))\sim\iota^{-m}\hat{\iota}_{m}(n)-1\label{eq:q_random_conv}
\end{equation}
for $m\in\{0,1,2\}$. Similarly, for $\tilde{\nu}$ defined by (\ref{eq:nu_def})
with $\tilde{\theta}=a_{\iota}$ and $\tilde{g}=g_{\iota}$ in (\ref{eq:logGW_tailmodel}),
substituting $y_{n}\lambda(1+\tilde{\nu}_{y_{n}}(y_{n}\lambda))$
for $y$ in (\ref{eq:multipl_error}) and using $s'\in RV_{\{\theta-1\}}$
and Theorems \ref{thm: PHI-omega-nu} and \ref{thm:Pickands-like_random},
almost surely,
\[
\frac{s(y_{n}\lambda(1+\hat{\nu}_{n}(y_{n}\lambda)))-s(y_{n}\lambda(1+\tilde{\nu}_{y_{n}}(y_{n}\lambda)))}{s'(y_{n})y_{n}}\sim\lambda{}^{\theta}h_{\theta}\left(\frac{1+\hat{\nu}_{n}(y_{n}\lambda)}{1+\tilde{\nu}_{y_{n}}(y_{n}\lambda)}\right)
\]
\begin{equation}
\sim\lambda^{\theta}(\hat{\nu}_{n}(y_{n}\lambda)-\tilde{\nu}_{y_{n}}(y_{n}\lambda))\label{eq:q_nu_random_conv}
\end{equation}
locally uniformly for $\lambda>0$. From (\ref{eq:al_def}) and (\ref{eq:def_ahat}),
using (\ref{eq:q_random_conv}), (\ref{eq:multipl_error}), $s'\in RV_{\{\theta-1\}}$
and (\ref{eq:iota_converges2}),
\[
(\hat{\theta}_{n}-a_{\iota}(y_{n}))\log\iota=\log\left(1+\frac{\mathcal{R}_{n}^{2}\frac{s'(y_{n}\iota^{2})\iota}{s'(y_{n}\iota)}-\mathcal{R}_{n}^{1}}{\frac{s(y_{n}\iota^{2})-s(y_{n}\iota)}{y_{n}\iota s'(y_{n}\iota)}}\right)-\log\left(1+\frac{\mathcal{R}_{n}^{1}\frac{s'(y_{n}\iota)\iota}{s'(y_{n})}-\mathcal{R}_{n}^{0}}{\frac{s(y_{n}\iota)-s(y_{n})}{y_{n}s'(y_{n})}}\right)
\]
\[
=\left(h_{\theta}(\iota)\right)^{-1}\left(\iota^{\theta}(\iota^{-2}\hat{\iota}_{2}(n)-1)(1+o(1))-(\iota^{-1}\hat{\iota}_{1}(n)-1)(1+o(1))\right.
\]
\begin{equation}
\left.-\iota^{\theta}(\iota^{-1}\hat{\iota}_{1}(n)-1)(1+o(1))+(\hat{\iota}_{0}(n)-1)(1+o(1))\right)\quad\textrm{a.s.}\label{eq:rho_randomerror}
\end{equation}

Because $1-F(X)$ has the uniform distribution on $(0,1)$, by \citet{Smirnov},
\begin{equation}
y_{n}(\hat{\iota}_{m}(n)-\iota^{m})\sqrt{k_{m}(n)}\overset{d}{\rightarrow}N(0,1)\quad\forall m\in\{0,1,2\}\label{eq:smirnov}
\end{equation}
as $n\rightarrow\infty$. Therefore, as $k_{2}(n)=o(k_{1}(n))$ and
$k_{1}(n)=o(k_{0}(n))$, (\ref{eq:rho_randomerror}) implies (\ref{eq:rho_asnormal}).
From (\ref{eq:def_qhat}) and (\ref{eq:logGW_tailmodel}),
\begin{equation}
\frac{\hat{s}_{n}(y_{n}\lambda)-\tilde{s}_{y_{n}}(y_{n}\lambda)}{y_{n}s'(y_{n})}=\mathcal{R}_{n}^{0}+\frac{h_{a_{\iota}(y_{n})}(\lambda)}{h_{a_{\iota}(y_{n})}(\iota)}\left(\mathcal{R}_{n}^{1}\frac{s'(y_{n}\iota)\iota}{s'(y_{n})}-\mathcal{R}_{n}^{0}\right)\label{eq:qhat-qtilde}
\end{equation}
\[
+\left(\frac{h_{\hat{\theta}_{n}}(\lambda)}{h_{\hat{\theta}_{n}}(\iota)}-\frac{h_{a_{\iota}(y_{n})}(\lambda)}{h_{a_{\iota}(y_{n})}(\iota)}\right)\left(\frac{s(y_{n}\iota)-s(y_{n})}{y_{n}s'(y_{n})}+\mathcal{R}_{n}^{1}\frac{s'(y_{n}\iota)\iota}{s'(y_{n})}-\mathcal{R}_{n}^{0}\right).
\]

As $s$ is increasing and $\sqrt{k_{2}(n)/k_{m}(n)}\log\log n\rightarrow0$
for $m\in\{0,1\}$, Lemma \ref{lem:random_argument_converges}(b)
in Subsection \ref{sub:Lemma} implies that
\begin{equation}
(\hat{\iota}_{m}(n)-\iota^{m})y_{n}\sqrt{k_{2}(n)}\rightarrow0\quad m\in\{0,1\}\quad a.s.\label{eq:ihat-i-van}
\end{equation}
and $(\hat{\iota}_{2}(n)-\iota^{2})y_{n}\sqrt{k_{2}(n)}=o(\log\log n)$
a.s., so by (\ref{eq:rho_randomerror}), $(\hat{\theta}_{n}-a_{\iota}(y_{n}))y_{n}\sqrt{k_{2}(n)}=o(\log\log n)$
a.s. Therefore, by Taylor's theorem (see (\ref{eq:kappa_def})),
\begin{equation}
\left|\frac{h_{\hat{\theta}_{n}}(\lambda)}{h_{\hat{\theta}_{n}}(\iota)}-\frac{h_{a_{\iota}(y_{n})}(\lambda)}{h_{a_{\iota}(y_{n})}(\iota)}-\kappa_{\theta}(\lambda,\iota)(\hat{\theta}_{n}-a_{\iota}(y_{n}))\right|=O(1)(\hat{\theta}_{n}-a_{\iota}(y_{n}))^{2}\quad a.s.\label{eq:h_over_h}
\end{equation}
locally uniformly in $\lambda>0$, with on the right-hand side (using
(\ref{eq:liminf_y})):
\begin{equation}
(\hat{\theta}_{n}-a_{\iota}(y_{n}))^{2}=o(1)(\log\log n)^{2}/(y_{n}^{2}k_{2}(n))=o(1)/(y_{n}k_{2}(n))\quad a.s.\label{eq:thetahat_err2}
\end{equation}

By (\ref{eq:q_random_conv}) and (\ref{eq:ihat-i-van}), $\mathcal{R}_{n}^{m}y_{n}\sqrt{k_{2}(n)}\rightarrow0$
a.s. for $m\in\{0,1\}$. Therefore, from (\ref{eq:qhat-qtilde}),
using (\ref{eq:h_over_h}), (\ref{eq:thetahat_err2}), (\ref{eq:multipl_error})
and $s'\in RV_{\{\theta-1\}}$, for all $\Lambda>1$,
\begin{equation}
\sup_{\lambda\in[\Lambda^{-1},\Lambda]}\left|\frac{\hat{s}_{n}(y_{n}\lambda)-\tilde{s}_{y_{n}}(y_{n}\lambda)}{y_{n}s'(y_{n})}-\kappa_{\theta}(\lambda,\iota)(\hat{\theta}_{n}-a_{\iota}(y_{n}))h_{\theta}(\iota)\right|y_{n}\sqrt{k_{2}(n)}\rightarrow0\quad a.s.\label{eq:s-s_normal}
\end{equation}

Therefore, by (\ref{eq:rho_asnormal}) and (\ref{eq:yn_over_logn}),
(\ref{eq:q_normal}) is obtained. Because $s$ is continuously increasing,
$s(y_{n}\lambda(1+\tilde{\nu}_{y_{n}}(y_{n}\lambda)))=\tilde{s}_{y_{n}}(y_{n}\lambda)$
and $s(y_{n}\lambda(1+\hat{\nu}_{n}(y_{n}\lambda)))=\hat{s}_{n}(y_{n}\lambda)$
in (\ref{eq:q_nu_random_conv}) so almost surely,
\[
\hat{\nu}_{n}(y_{n}\lambda)-\tilde{\nu}_{y_{n}}(y_{n}\lambda)=(1+o(1))\lambda^{-\theta}\frac{\hat{s}_{n}(y_{n}\lambda)-\tilde{s}_{y_{n}}(y_{n}\lambda)}{y_{n}s'(y_{n})}
\]
locally uniformly in $\lambda>0$. Therefore, using (\ref{eq:yn_over_logn})
and (\ref{eq:q_normal}), we obtain (\ref{eq:nu_normal}). This proves
Theorem \ref{thm:logGW_normal}. 

To prove Corollary \ref{cor:normality}, note that (\ref{eq:conv_rate})
must hold locally uniformly in $\lambda\geq1$: with $r$ defined
by $r(y):=\log s'(y)-(\theta-1)\log y$, (\ref{eq:conv_rate}) is
equivalent to $\lim_{y\rightarrow\infty}(r(y\lambda)-r(y))\phi(y)=0$
for all $\lambda\geq1$, which holds locally uniformly in $\lambda\geq1$
by Theorem 3.1.7c of \citet{Bingham}. By integration,
\begin{equation}
s(y\lambda)-s(y)=ys'(y)h_{\theta}(\lambda)(1+o(1)/\phi(y))\label{eq:ERV-fast}
\end{equation}
locally uniformly in $\lambda\geq1$. Therefore,
\begin{equation}
a_{\iota}(y)=\theta+o(1)/\phi(y),\label{eq:a-fast}
\end{equation}
so by the mean value theorem, $\frac{h_{a_{\iota}(y)}(\lambda)}{h_{a_{\iota}(y)}(\iota)}-\frac{h_{\theta}(\lambda)}{h_{\theta}(\iota)}=O(a_{\iota}(y)-\theta)=o(1)/\phi(y)$
locally uniformly in $\lambda\geq1$. Using (\ref{eq:ERV-fast}),
therefore, 
\begin{equation}
\frac{\tilde{s}_{y}(y\lambda)-s(y\lambda)}{ys'(y)}=\frac{s(y\iota)-s(y)}{ys'(y)}\Bigl(\frac{h_{\theta}(\lambda)}{h_{\theta}(\iota)}+o(1)/\phi(y)\Bigr)+\frac{s(y)-s(y\lambda)}{ys'(y)}=o(1)/\phi(y)\label{eq:s_tilde_fast}
\end{equation}
locally uniformly in $\lambda\geq1$. Finally, by (\ref{eq:multipl_error})
and Theorem \ref{thm: PHI-omega-nu}, as $s'\in RV_{\{\theta-1\}}$,
\begin{equation}
s(y\lambda(1+\tilde{\nu}_{y}(y\lambda)))-s(y\lambda)\sim\lambda^{\theta}h_{\theta}\left(1+\tilde{\nu}_{y}(y\lambda)\right)ys'(y)\sim\lambda^{\theta}\tilde{\nu}_{y}(y\lambda)ys'(y)\label{eq:s_nu_fast}
\end{equation}
locally uniformly in $\lambda\geq1$. Since $s$ is continuously increasing,
$s(z(1+\tilde{\nu}_{y}(z)))=\tilde{s}_{y}(z)$ for all $z>0$, so
combining (\ref{eq:s_nu_fast}) and (\ref{eq:s_tilde_fast}), it follows
that
\begin{equation}
\tilde{\nu}_{y}(y\lambda)=o(1)/\phi(y)\label{eq:nu_fast}
\end{equation}
locally uniformly in $\lambda\geq1$. Using $k_{2}(n)=O(\phi^{2}(y_{n})y_{n}^{-2})$,
(\ref{eq:rho_asnormal-1}) follows from (\ref{eq:rho_asnormal}) and
(\ref{eq:a-fast}); using (\ref{eq:yn_over_logn}) as well, (\ref{eq:q_normal-1})
follows from (\ref{eq:q_normal}) and (\ref{eq:s_tilde_fast}), and
(\ref{eq:nu_normal-1}) follows from (\ref{eq:nu_normal}) and (\ref{eq:nu_fast}).

\subsection{\label{sub:Proof-lograte}Proof that (\ref{eq:log_rate}) implies
$q\in ERV_{\{1\}}$}

Take $w\in RV_{\{0\}}$. With $R_{\lambda}(t):=(U(t\lambda)-U(t))/w(t)$,
(\ref{eq:log_rate}) implies $w(t\lambda)/w(t)=(R_{\lambda\xi}(t)-R_{\lambda}(t))/R_{\xi}(t\lambda)=1+o(1/\log t)$
for all $\lambda\geq1$ and $\xi>1$, so by \citet{B=000026S} (see
\citet{Bingham}, Theorem 2.3.1), $w(t^{\lambda})/w(t)\rightarrow1$
locally uniformly in $\lambda\geq1$ as $t\rightarrow\infty$; applying
Theorem 3.6.6 in \citet{Bingham} gives $U(t^{\lambda})-U(t)\sim(\lambda-1)w(t)\log t$
for all $\lambda\geq1$, so $q\in ERV_{\{1\}}$.

\subsection{Lemmas\label{sub:Lemma}}
\begin{lemma}
\label{lem:f-logf}Let $f$ be a nondecreasing function satisfying
$f(\infty)>0$. 

(a) If $f\in ERV_{\{\theta\}}$, then $\log f\in ERV_{\{\min(\theta,0)\}}(g)$
with the positive function $g$ converging to $\max(\theta,0)$. 

(b) If $\log f\in ERV_{\{\theta\}}(g)$, then $f\in ERV$ if and only
if $g$ converges to some $g_{\infty}\in[0,\infty)$. If so, then
$f\in ERV_{\{\min(\theta,0)+\max(0,g_{\infty})\}}(fg)$.

(c) For $\theta<0$, $\log f\in ERV_{\{\theta\}}$ if and only if
$f\in ERV_{\{\theta\}}$.\end{lemma}
\begin{svmultproof}
\noindent If $f\in ERV_{\{\theta\}}$ with $\theta>0$, then $f\in RV_{\{\theta\}}$
so $\log q\in ERV_{\{0\}}(\theta)$. If $f\in ERV_{\{\theta\}}(\bar{g})$
with $\theta\leq0$, then as $y\rightarrow\infty$, $\bar{g}(y)/f(y)\rightarrow0$
(see \citet{Laurens  boek}, Lemma 1.2.9). Therefore, for every $\lambda\in(0,1)\cup(1,\infty)$,
also $f(y\lambda)/f(y)-1\rightarrow0$, so $\log f(y\lambda)-\log f(y)\sim f(y\lambda)/f(y)-1$
and as $\bar{g}(y)/f(y)\rightarrow0$, we obtain $\log f\in ERV_{\{\theta\}}(\bar{g}/f)$,
proving (a). 

\noindent If $g$ converges to $g_{\infty}>0$, then $f\in RV_{\{g_{\infty}\}}$,
so \textit{$f\in ERV_{\{g_{\infty}\}}(fg)$. }If $g$ converges to
$0$, then for every $\lambda\in(0,1)\cup(1,\infty)$, $f(y\lambda)/f(y)-1\rightarrow0$,
so $f(y\lambda)/f(y)-1\sim\log f(y\lambda)-\log f(y)$ as $y\rightarrow\infty$.
Therefore, \textit{$f\in ERV_{\{\theta\}}(fg)$} and necessarily,
$\theta\leq0$. This proves the ``if'' part of (b); the ``only
if'' part follows from (a), and (c) follows directly from (a) and
(b). \end{svmultproof}

\begin{lemma}
\label{lem:random_argument_converges}(a) $\hat{\iota}$ defined by
(\ref{eq:iota_def}) with (\ref{eq:k_j}) satisfies
\begin{equation}
X_{n-k_{m}(n)+1,n}=q(y_{n}\hat{\iota}_{m}(n))\quad\forall m\in\{0,1,2\},\: n\in\mathbb{N}\quad\textrm{a.s.}\label{eq:X=00003Dq_yiota}
\end{equation}

(b) Let $k_{2}:\mathbb{\mathbb{N}\rightarrow N}$ satisfy (\ref{eq:k2_cond})
and $k_{2}(n)/\log\log n\rightarrow\infty$. If $q\in ERV$, then
\begin{equation}
\hat{\iota}_{m}(n)\rightarrow\iota^{m}\quad\forall m\in\{0,1,2\}\quad\textrm{a.s.}\label{eq:iota_converges2}
\end{equation}

If $F$ is continuous, then
\begin{equation}
(\hat{\iota}_{m}(n)-\iota^{m})y_{n}\sqrt{k_{m}}/\log\log n\rightarrow0\quad\forall m\in\{0,1,2\}\quad\textrm{a.s.}\label{eq:iota_converges3}
\end{equation}
\end{lemma}
\begin{svmultproof}
Almost surely, $X_{n-k+1,n}=q(-\log\mathcal{U}_{k,n})$ for all $n\in\mathbb{N}$
and $k\in\{1,...,n\}$, with $\mathcal{U}_{k,n}$ the $k^{th}$ order
statistic of a sample of $n$ independent random variables uniformly
distributed on $(0,1)$. Therefore, by (\ref{eq:iota_def}),
\begin{equation}
\hat{\iota}_{m}(n)y_{n}=q^{-1}(q(-\log\mathcal{U}_{k_{m}(n),n}))\quad\forall m\in\{0,1,2\},\: n\in\mathbb{N}\quad a.s.\label{eq:iota_hat}
\end{equation}
and (\ref{eq:X=00003Dq_yiota}) follows. For (b), note that $k_{m}(n)/n\rightarrow0$
and $k_{m}(n)/\log\log n\rightarrow\infty$ for each $m\in\{0,1,2\}$,
so by \citet{Einmahl =000026 Mason} (Theorem 3(III) with $\nu=\frac{{\scriptstyle 1}}{2}$),
\[
\left((n/k_{m}(n))\mathcal{U}_{k_{m}(n),n}-1\right)\sqrt{k_{m}}/\log\log n\rightarrow0\quad\forall m\in\{0,1,2\}\quad\textrm{a.s.}
\]
and as (\ref{eq:k2_cond}) implies that $(\log(n/k_{m}(n))-\iota^{m}y_{n})=O(1/k_{m}(n))$
for $m=0,1,2$,
\begin{equation}
(\log\mathcal{U}_{k_{m}(n),n}+\iota^{m}y_{n})\sqrt{k_{m}}/\log\log n\rightarrow0\quad\forall m\in\{0,1,2\}\quad\textrm{a.s.}\label{eq:unifproc_conv}
\end{equation}

If $F$ is continuous, then $q^{-1}\circ q=\textrm{Id}$, so (\ref{eq:iota_converges3})
follows from (\ref{eq:iota_hat}) and (\ref{eq:unifproc_conv}). If
not, then $\hat{\iota}_{m}(n)=y_{n}^{-1}q^{-1}(q(\iota^{m}y_{n}+o(1)))$
a.s. by (\ref{eq:iota_hat}) and (\ref{eq:unifproc_conv}), so if
$q\in ERV$, then (\ref{eq:iota_converges2}) follows from (\ref{eq:F_exp_logform}). \end{svmultproof}

\begin{acknowledgements}
The author would like to thank John Einmahl, Laurens de Haan, Juan-Juan
Cai, two anonymous Referees and especially an anonymous Associate
Editor of EXTREMES for their helpful criticism and suggestions.
\end{acknowledgements}

\noindent \begin{flushleft}
\textbf{\small{}Conflict of Interest }{\small{}The author declares
that he has no conflict of interest.}
\par\end{flushleft}{\small \par}

\end{document}